\documentclass[11pt]{amsart}
\usepackage{amssymb,amsmath, amsthm}
\usepackage{a4wide}
\usepackage[colorlinks]{hyperref}
\usepackage[all,color]{xy}
\usepackage{comment}
\usepackage{enumerate}
\usepackage{color}
\usepackage{textcomp}
\usepackage{graphicx}
\usepackage[section]{placeins}
\usepackage{url}
\author{Ulf K\"uhn}
\address{Fachbereich Mathematik, Universit\"{a}t Hamburg, Bundesstrasse 55,
20146 Hamburg, Germany}
\email{kuehn@math.uni-hamburg.de }
\author{J. Steffen M\"uller}
\address{Institut f\"ur Mathematik, Carl von Ossietzky
Universit\"{a}t Oldenburg, 26111 Oldenburg, Germany}
\email{jan.steffen.mueller@uni-oldenburg.de }

\title[Constructing elements of $K_2$ of curves]{A geometric approach to constructing elements of $K_2$ of curves}
\newtheorem{thm}{Theorem}[section]
\newtheorem{prop}[thm]{Proposition}
\newtheorem{lemma}[thm]{Lemma}
\newtheorem{cor}[thm]{Corollary}
\newtheorem{conj}[thm]{Conjecture}

\theoremstyle{remark}
\newtheorem{rk}[thm]{Remark}
\newtheorem{constr}[thm]{Construction}
\newtheorem{defn}[thm]{Definition}
\newtheorem{ex}[thm]{Example}
\newcommand\Q{\mathbb{Q}}
\newcommand\C{\mathbb{C}}

\newcommand\Z{\mathbb{Z}}
\newcommand\N{\mathbb{N}}
\newcommand\A{\mathbb{A}}
\newcommand\R{\mathbb{R}}
\newcommand\p{\mathfrak{p}}

\newcommand\Spec{\mathop{\rm Spec}\nolimits}

\renewcommand\O{\mathcal{O}}

\newcommand\ord{\mathop{\rm ord}\nolimits}

\renewcommand{\div}{\operatorname{div}}

\newcommand{\disc}{\operatorname{disc}}

\newcommand{\Pic}{\operatorname{Pic}}

\newcommand{\BP}{{\mathbb P}}

\newcommand{\Ka}{\overline{K}}
\newcommand{\Qa}{\overline{\Q}}

\newcommand{\KC}{\mathcal{C}}

\begin{document}
 
\date{\today}
\begin{abstract}
  We present a framework for constructing examples of smooth projective curves over number
  fields with explicitly given elements in their second $K$-group using elementary algebraic geometry.
  This leads to new examples for hyperelliptic curves and smooth plane quartics.
  Moreover, we show that most previously known constructions can be reinterpreted using our
  framework. 
\end{abstract}
\maketitle
\section{Introduction}\label{sec:intro}
The Beilinson conjectures are of fundamental importance in algebraic $K$-theory and
arithmetic geometry, predicting a relation between special values of $L$-functions and
regulators of certain higher $K$-groups of smooth projective varieties defined over number
fields.
See~\cite{schneider} for an introduction.

Let $C$ denote a smooth projective geometrically irreducible curve of genus $g$ defined over a number field $K$ with
ring of integers $\O_K$. 
We denote by $K_2^T(C)$ the tame second $K$-group of $C$, defined in Section~\ref{sec:k2}.
A special case of Beilinson's conjecture predicts that 
that a certain subgroup $K_2(C;\,\O_K)$ of $K_2^T(C)/\mathrm{torsion}$ is free of rank
$g\cdot[K:\Q]$.
In order to prove this conjecture or at least test it numerically in examples, one needs a method to
come up with enough independent elements of $K_2(C;\,\O_K)$.
In general it is quite difficult to construct elements of $K_2^T(C)$ (not to mention
$K_2(C;\, \O_K))$ for a given curve $C$.
Apart from the work of 
Beilinson~\cite{beilinson} (for modular curves over abelian number fields) and
Deninger~\cite{deninger} (for elliptic curves with complex multiplication)
no systematic constructions are known to date.
Instead, a number of ad hoc approaches have been developed, see for
instance~\cite{bloch-grayson},~\cite{SR},~\cite{ddz} and~\cite{dl}.
These produce certain families of curves for which it is known that many elements of
$K^T_2(C)$ exist.

In this note we present a geometric approach to constructing algebraic curves $C$
together with elements in $K^T_2(C)$ using elementary algebraic geometry.
Our idea is as follows: We first choose plane curves 
\[
  C_1,\ldots, C_m \subset \BP^2_K
\]
of respective degrees $d_1, \ldots, d_m$; 
then we consider functions $f_{kl}$ on $\BP^2_K$ such that 
\[
 \div ( f_{kl} ) = d_l \cdot C_k -  d_k \cdot C_l.
\]
These are determined by the equations of the curves $C_1,\ldots,C_m$ up to scaling. 
We then construct a plane curve $C/K$ such that many of the classes 
\begin{align}\label{functions}
    \big \{ f_{kl}\big |_C, f_{k'l'}\big |_C \big \}   \in K_2(K(C))
\end{align}
represent elements in $K_2^T(C)/\mathrm{torsion}$ after a suitable scaling of the functions $f_{kl}$ and
$f_{k'l'}$.
We require that the curves $C_i$ intersect the curve $C$ in very few points and that 
the intersection multiplicities 
between $C$ and $C_i$ at these points satisfy certain requirements. This yields
a system of equations for the coefficients of $C$ that in many cases has a solution.
In these cases, we then change our point of view and treat the coefficients of the curves
$C_i$ and the coordinates of the points of intersection as indeterminates.  
Using this approach, we finally get a parametrization of a family of curves
such that every curve in this family has a number of known representatives of elements in
$K^T_2(C)/\mathrm{torsion}$.

In fact, we show that we can reinterpret most known constructions using our method and we can
also come up with some new families of smooth plane quartics and hyperelliptic curves.
We focus on the geometric aspects, so we only discuss integrality of our elements
in some examples (for some hyperelliptic curves and smooth plane quartics) and
we touch upon independence of the elements only briefly.
It would be an interesting project to check independence in the families constructed in
Sections~\ref{sec:spq},~\ref{sec:conics} and~\ref{sec:nekovar}, for instance using a limit
argument as in~\cite{deJeu:limit} or~\cite{dl}.

\begin{rk}\label{rk:dl}
  In recent work~\cite{dl} de Jeu and Liu manage to construct $g$ independent  elements of
  $K_2^T(C)$ for certain families of curves, including the hyperelliptic examples
  of~\cite{ddz} discussed in Section~\ref{sec:hyp}, but also non-hyperelliptic ones.
  Furthermore, they show that in some of these families, all of the elements are in fact
  integral.
  We would like to point out that their construction also fits into our general framework:
  They begin with a number of lines in $\A^2$; 
  their curves are then given as the smooth projective model of a certain
  singular affine curve constructed using these lines.
\end{rk}

The organization of the paper is as follows:
In Section~\ref{sec:k2} we give a rather gentle and elementary introduction to $K_2$ of
curves, closely following~\cite{ddz}.
The constructions in~\cite{bloch-grayson} and~\cite{ddz} work for (hyper-) elliptic curves 
and use torsion divisors, following an approach which goes back to work of Bloch and which
we recall in Section~\ref{sec:torsion}.
Then we show in Section~\ref{sec:hyp} that our method gives a simple way to obtain 
many examples of hyperelliptic curves with many explicitly given elements, including
the examples of~\cite{ddz}. 
Still using torsion divisors, we apply our approach to smooth plane quartics in
Section~\ref{sec:spq}, where the curves $C_i$ are lines, and in Section~\ref{sec:conics}, where
the curves $C_i$ are conics or lines.
Finally, we show in Section~\ref{sec:nekovar} that we are not restricted to torsion divisors 
 by generalizing a construction for elliptic curves discussed in~\cite{SR}, where it is
attributed to Nekov\'a\v{r}, to certain curves of higher genus.

We would like to thank Hang Liu for helpful conversations and Rob de Jeu for helpful
conversations and a careful reading of an earlier version of this paper, which resulted in
many improvements, in particular the ``mixed'' case of Proposition~\ref{prop:hyp_ex_even},
see Remark~\ref{rk:mixed}.
The second author was supported by DFG-grant KU 2359/2-1.

\section{$K_2$ of curves and Beilinson's conjecure}\label{sec:k2}
We give a brief down-to-earth introduction to Beilinson's conjecture on $K_2$ of curves
over number fields, following the discussion of~\cite[\S3]{ddz}.
The definition of higher $K$-groups of $C$ is due to Quillen~\cite{quillen} and is rather
involved.
However, Beilinson's original conjecture can be formulated in terms of the
tame second $K$-group $K_2^T(C)$ of $C$, whose definition is quite simple in comparison,
and we follow this approach here.

Let $C$ be a smooth, projective, geometrically irreducible curve of genus $g$
defined over a number field $K$ with ring of integers $\O_K$ and fixed algebraic closure
$\Ka$.
By Matsumoto's Theorem~\cite[Theorem~11.1]{milnor}, the second $K$-group of the
field $K(C)$ is given by
\[
    K_2(K(C)) = K(C)^\times\otimes_\Z K(C)^\times/\langle f\otimes (1-f),\; f \in
    K(C)^\times\setminus\{1\}\rangle.
\]
If $f,h \in K(C)^\times$, then we write $\{f, h\}$ for the class of $f
\otimes h$ in $K_2(K(C))$.
Hence $K_2(K(C))$ is the abelian group with generators $\{f,h\}$ and relations
\begin{align*}
  \{f_1f_2,h\}&=\{f_1,h\}+\{f_2,h\}\\
 \{f,h_1,h_2\}&=\{f,h_1\}+\{f, h_2\}\\
    \{f, 1-f\}&=0\textrm{ for } f \in K(C)^{\times}\setminus\{1\}.
\end{align*}
  For $P$ in the set $C^{(1)}$ of closed points of $C$ 
  and $\{f, h\} \in K_2(K(C))$, we define the {\em tame symbol}
\[
    T_P(\{f, h\}) = (-1)^{\ord_P(f)\ord_P(h)}
    \frac{f^{\ord_P(h)}}{h^{\ord_P(f)}}(P)\in K(P)^\times,
\]
and we extend this to $K_2(K(C))$ by linearity.
Then we have the product formula
\begin{equation}\label{eq:prod_form}
  \prod_{P \in C^{(1)}} N_{{K(P)}/K}\left(T_P(\alpha)\right) = 1 \quad\textrm{ for
    all}\;\alpha \in K_2(K(C)),
\end{equation}
which generalizes Weil's reciprocity law, see~\cite[Theorem~8.2]{bass}.
Setting
\[T = \prod_{P \in C^{(1)}} T_P,\]
we define the {\em tame second $K$-group}
\[
  K_2^T(C) = \ker\left(T:K_2(K(C)) \to \bigoplus_{P \in C^{(1)}}
    K(P)^\times\right)
\]
of $C$.
Note that in general $K^T_2(C)$ differs from the second $K$-group $K_2(C)$
associated to $C$ by Quillen in~\cite{quillen}.
However, we get an exact sequence
\[
\xymatrix{
  \bigoplus_{P \in C^{(1)}}K_2(K(P)) \ar[r]& K_2(C) \ar[r]& K_2(K(C)) \ar[r]^-{T}&
  \bigoplus_{P \in C^{(1)}}K(P)^\times
}
\]
from~\cite{quillen} and since $K_2$ of a number field is torsion~\cite{garland}, this implies
\[
    K^T_2(C)/\textrm{torsion} \cong K_2(C)/\textrm{torsion}.
\]
As our motivation is Beilinson's conjecture on $K_2(C)$, which does
not depend on the torsion subgroup of $K_2(C)$ at all, we only discuss
$K_2^T(C)$.

By~\cite{bloch-grayson}, the correct $K$-group for the statement of
Beilinson's conjecture is not the quotient $K^T_2(C)/\textrm{torsion}$, but a subgroup
$K_2(C;\O_K)$ of $K^T_2(C)/\textrm{torsion}$ defined by certain  integrality conditions.
Let $\KC\to\Spec(\O_K)$ be a proper regular model of $C$.
For an irreducible component $\Gamma$ of a special fiber $\KC_\p$ of $\KC$ and $\{f,
h\} \in K_2^T(C)$ we define
\[
    T_\Gamma(\{f, h\}) = (-1)^{\ord_\Gamma(f)\ord_\Gamma(h)}
    \frac{f^{\ord_\Gamma(h)}}{h^{\ord_\Gamma(f)}}(\Gamma) \in
    k_\p(\Gamma)^\times,
\]
where $k_\p(\Gamma)$ is the function field of $\Gamma$ and $\ord_\Gamma$ is the normalized discrete
valuation on the local ring $\O_{\KC,\Gamma}$, extended to its field of fractions $K(C)$.
We extend $T_\Gamma$ to $K_2^T(C)$ by linearity and set
\[
    K_2^T(\KC) = \ker\left(K^T_2(C) \to
    \bigoplus_{\p\subset \O_K}\bigoplus_{\Gamma \subset \KC_\p}
    k_\p(\Gamma)^\times\right).
\]
Here the $\Gamma$-component of the map is given by $T_\Gamma$ and the
sums run through the finite primes $\p$ of $\O_K$ and the irreducibe
components of $\KC_\p$, respectively.

One can show (see~\cite[Proposition~4.1]{dl}) that $K^T_2(\KC)$ does
not depend on the choice of $\KC$, so that 
\[
    K_2(C;\O_K) := K_2^T(\KC)/\textrm{torsion}
\]
is well-defined.
We call an element of $K_2^T(C)/\textrm{torsion}$ {\em integral} if it
lies in $K_2(C;\O_K)$.

Finally, we can state a weak version of Beilinson's conjecture (as modified by Bloch).

\begin{conj}\label{conj:beilinson}
    Let $C$ be a smooth, projective, geometrically irreducible curve of
    genus $g$ defined over a number field $K$.
    Then the group $K_2(C;\O_K)$ is free of rank $g\cdot[K:\Q]$.
\end{conj}
Finite generation of $K_2(C;\O_K)$ was already conjectured by Bass.
The full version of Beilinson's conjecture on $K_2$ can be found in~\cite{ddz}.
It relates the special value of the Hasse-Weil $L$-function of $C$ at $s=2$ to the
determinant of the matrix of values of a certain regulator pairing
\[
  H_1(C(\C), \Z) \times K^T_2(C)/\textrm{torsion}\to\R
\]
at the elements of a basis of the anti-invariants of $H_1(C(\C),\Z)$ under complex
conjugation and at the elements of a basis of $K_2(C,\O_K)$.

However, even Conjecture~\ref{conj:beilinson} is still wide open in general,
largely due to the fact that it is rather difficult to construct elements of $K_2$.
Therefore new methods for such constructions are needed.

\section{Torsion construction}\label{sec:torsion}
Most of the known constructions of elements of $K_2$ of curves use {\em torsion
divisors}, in other words, divisors on the curve of degree zero whose divisor 
classes in the Jacobian of the curve are torsion.
For simplicity, we restrict to the case $K=\Q$; all constructions can
be extended to arbitrary number fields using straightforward
modifications.

We first recall the following construction from~\cite{ddz}, originally due to Bloch:
\begin{constr}\label{constr:torsion}(\cite[Construction~4.1]{ddz})
    Let $C/\Q$ be a smooth, projective, geometrically irreducible curve, let $h_1,h_2,h_3
    \in \Q(C)^*$ and let $P_1,P_2,P_3 \in C(\Q)$ be such that
    \[
        \div(h_i) = m_i(P_{i+1}) - m_i(P_{i-1}),\quad i \in \Z/3\Z,
    \]
    where $m_i \in \N$ is the order of the class of $(P_{i+1}) - (P_{i-1})$ in $\Pic^0(C)$.
    Then we define symbols 
    \[
        S_i =\left\{\frac{h_{i+1}}{h_{i+1}(P_{i+1})},
        \frac{h_{i-1}}{h_{i-1}(P_{i-1})}\right\} \in K_2(\Q(C)),\quad i \in \Z/3\Z.
    \]
\end{constr}

We summarize the most important facts about Construction~\ref{constr:torsion}:

\begin{prop}\label{prop:constr_props}
    Keep the notation of Construction~\ref{constr:torsion}.
    \begin{enumerate}[(i)]
        \item The symbols $S_i$ are elements of $K^T_2(C)$.
        \item There is a unique element $\{P_1,P_2,P_3\} \in K^T_2(C)/\mathrm{torsion}$ such that
            \[
                S_i = \frac{\mathrm{lcm}(m_1,m_2,m_3)}{m_i}\{P_1,P_2,P_3\}, \quad i \in \{1,2,3\}
            \]
            in $K^T_2(C)/\mathrm{torsion}$.
          \item  If $\sigma\in \mathfrak{S}_n$ is an even permutation, then we have $\{P_{\sigma(1)},P_{\sigma(2)},P_{\sigma(3)}\} =
            \{P_1,P_2,P_3\}$. If $\sigma\in \mathfrak{S}_n$ is odd, then
            $\{P_{\sigma(1)},P_{\sigma(2)},P_{\sigma(3)}\} =
            -\{P_1,P_2,P_3\}$. 
        \item Suppose that there exists, in addition, a point $P_4 \in
            C(\Q)$ such that all
            differences $(P_i) - (P_j)$ are torsion divisors for $i,j \in \{1,\ldots,4\}$. 
            Then the following elements are linearly
            dependent in $K^T_2(C)/\mathrm{torsion}$:
   \[
   \{P_1,P_2,P_3\},\;\{P_1,P_2,P_4\},\;\{P_1,P_3,P_4\},\; \{P_2,P_3,P_4\}
   \]
   \end{enumerate}
\end{prop}
\begin{proof}
    All assertions are stated and proved in~\cite[\S4]{ddz}.
\end{proof}

\begin{rk}\label{rk:tors}
We stress that there are curves which have no rational torsion points in
their Jacobian, so we cannot expect that the torsion construction will always suffice to
construct enough elements of $K_2(C,\Z)$.
\end{rk}

We will use the following notation:
 \begin{defn} \label{m-contact}
     If $C,\,D\subset \BP^2(K)$ are plane curves defined over a field $K$ and $m$ is a positive integer,
     then we call a 
     point $P \in C(\Ka) \cap D(\Ka)$ an {\em $m$-contact point} between $C$ and $D$ 
if the intersection multiplicity $I_P(C,D)$ of $C$ and $D$ at $P$ is equal to $m$.
If $C$ has degree $d_C$ and $D$ has degree $d_D$, then we call $P \in C(\Ka)$ a {\em maximal
contact point} between $C$ and $D$ if $P$ is a $d_Cd_D$-contact point between $C$ and $D$.
\end{defn}
\begin{rk}
It is possible for a point to be an $m_i$-contact point for several distinct positive integers $m_i$.
\end{rk}
 
\begin{constr}\label{constr:basic}
Now we describe the use of Construction~\ref{constr:torsion} in our setup.
We work in the projective plane $\BP^2$ throughout.
We fix a rational point $\infty\in \BP^2$ and a line $L_\infty$ through this point.
We consider those smooth plane curves $C$ such that $L_\infty$ meets $C$ precisely in
$\infty$.

Then we choose $m\ge 2$ rational points $P_1,\ldots,P_m$ distinct from $\infty$ and require that $C$ meets the
lines $L_{P_i}$ through $P_i$ and $\infty$ precisely in these two points.
By abuse of notation, the defining polynomial of a line $L$ will also be denoted by $L$.
Then each quotient $L_{P_i}/L_\infty$ determines a function $h_i \in k(C)$ with divisor 
\[\div(h_i)= m_i (P_i) - m_i (\infty)  \in Z^1(C)\]
for some positive integer $m_i$.
It follows that for $j \ne i$ the quotients $L_{P_i}/L_{P_j}$ also define torsion
divisors.
Hence all functions obtained in this way are as in Construction~\ref{constr:torsion}, thus we get
elements \[\{\infty,P_i,P_j\} \in K_2^T(C)/\mathrm{torsion}.\] 

In some cases we cannot get as many non-trivial elements
as we want using just one point of maximal contact because of obvious relations, see for
instance~\cite[Example~5.2]{ddz} or the discussion following Lemma~\ref{lemma:quartic_vertical}
below. 
To remedy this, we can require that there is an additional rational point $O\ne \infty$ on $C$ such that there is
a curve $D$ which meets $C$ in precisely this point.
Then, with abuse of notation as above, the quotient $D/L_{\infty}^{\deg(D)}$ 
determines a function $h \in k(C)$ with divisor 
\[\div(h)= \deg(D)\deg(C)\left( (O) -  (\infty)\right)  \in Z^1(C).\]
As above, we find that the divisors $(O) - (P_i)$ are torsion divisors as well.
Therefore, again using Construction~\ref{constr:torsion},  we get further elements
$\{\infty, O, P_i\}  \in K_2^T(C)/\mathrm{torsion}$ and
these might be nontrivial (in this case it suffices to have $m\ge 1$).

Moreover, we can add the requirement that there are rational points $Q_j$ such that the line
through $Q_j$ and $O$ intersects $C$ precisely in these two points, leading to the further
elements $\{\infty, Q_i, Q_j\}$ and $\{\infty, O, Q_j\}$ as before and, finally,
$\{\infty, P_i, Q_j\}$. 
\end{constr}

\begin{rk}\label{rk:prod_sing}
  We can also use plane curves with a singularity of a particular kind.
  Namely, let $C'$ be a plane curve which is smooth outside a rational point $\infty'$ 
  of maximal contact between $C'$ and a line $L_{\infty'}$, with the property that 
  there is a unique point $\infty$ above $\infty'$ in the normalization $C$ of $C'$.
  Then there is a bijection between the closed points on $C$ and those on $C'$.
  It follows that we can work on the singular curve $C'$ as described in
  Construction~\ref{constr:basic} to obtain elements in $K_2^T(C)/\mathrm{torsion}$ as above.
\end{rk}

\section{Hyperelliptic curves}\label{sec:hyp}
The Beilinson conjecture on $K_2$ of 
hyperelliptic curves was studied by Dokchitser, de Jeu and Zagier in~\cite{ddz}. 
They considered several families of hyperelliptic curves $C/\Q$ which possess at least $g(C)$ elements of $K_2(C;\Z)$
using Construction~\ref{constr:torsion}.  
We first recall their approach; then we describe how it can be
viewed in terms of our geometric interpretation.
This leads to a much simpler way of constructing their families (and
others).
Finally, we discuss integrality of the elements obtained in this manner.

Consider a hyperelliptic curve $C/\Q$ of degree $d \in \{2g+1, 2g+2\}$
with a $\Q$-rational Weierstrass point $\infty$ and a point $O \in
C(\Q)$ such that $(O) - (\infty)$ is a $d$-torsion divisor.
Then, if $P$ is another $\Q$-rational Weierstrass point, the divisor
$(P) - (\infty)$ is~2-torsion and we can apply
Construction~\ref{constr:torsion} to find an element $\{\infty, O,
P\} \in K_2^T(C)/\mathrm{torsion}$.

Using Riemann-Roch one can show~\cite[Examples~5.3,~5.6]{ddz} that such a curve $C$ has an affine model
\begin{equation}\label{eq:ddz_model}
  F(x,y) = y^2 + f_1(x)y +x^d = 0,
\end{equation}
such that $\infty$ is the unique point at infinity on $C$ and $O = (0,0)$.
Here
\begin{equation}\label{eq:f1}
    f_1(x) = b_0+b_1x+\ldots+b_gx^g+\delta x^{2g+1} \in \Q[x],\quad 
    \delta = \left\{ \begin{array}{cc} 0, & d=2g+1\\ 2, &
    d=2g+2\end{array} \right.
\end{equation}
and $b_g \ne 0$ if $d=2g+1$; 
moreover, we have $\mathrm{disc}(-4x^d+f_1(x)^2)\ne0$.

The results of~\cite{ddz} rely on the following result, see~\cite[\S6]{ddz}:
\begin{prop}\label{prop:ddz}(Dokchitser-de Jeu-Zagier)
    Let $C$ be the hyperelliptic curve associated to an affine
    equation~\eqref{eq:ddz_model}.
    Let $\alpha\in \Q$ be a root of $t(x) = -x^d+f_1(x)^2/4$ and let
    $P = (\alpha, -f_1(\alpha)/2) \in C(\Q)$. 
    Then $\infty, O $ and $P$ define an element    
    $\{\infty, O, P\}$ of  $ K^T_2(C)/\mathrm{torsion}$.
\end{prop}
\begin{proof}
Via the coordinate change  $(x,y) \mapsto (x,y-f_1(x)/2)$, the curve $C$ is isomorphic
to the hyperelliptic curve with affine equation
\begin{equation}\label{eq:hyperelliptic_model}
    y^2 -t(x) = 0,
\end{equation}
where 
\[
  t(x) = -x^d+f_1(x)^2/4.
\]
On this model, the affine
Weierstrass points are of the form $(\alpha, 0)$, where $t(\alpha) =
0$.
So if $\alpha \in \Q$ is a root of $t$, then $P = (\alpha,
-f_1(\alpha)/2) \in C(\Q)$ is a rational Weierstrass point and the result follows from
the discussion preceding the theorem.
\end{proof}
\begin{rk}\label{rk:factors}
In~\cite[\S7,10]{ddz} suitable polynomials were constructed using clever
manipulation of polynomials and brute force computer searches.
More precisely, Dokchitser-de Jeu-Zagier show that every $\mathrm{Gal}(\Qa/\Q)$-orbit $\{P_\sigma\}_{\sigma}$ of
Weierstrass points, and therefore every $\mathrm{Gal}(\Qa/\Q)$-orbit
$\{\alpha_\sigma\}_{\sigma}$ of roots of $t(x)$, defines an element of
$K^T_2(C)/\mathrm{torsion}$.
Hence the irreducible factors of $t(x)$ determine the number of
elements in $K^T_2(C)/\mathrm{torsion}$ one can obtain in this way. 
The crucial, and most difficult, step is therefore the construction of polynomials
$f_1(x)$ as above such that $t(x)$ has many rational factors.  
\end{rk}

The results from~\cite{ddz} as stated in Proposition~\ref{prop:ddz} fit into
Construction~\ref{constr:basic}, refined according to Remark~\ref{rk:prod_sing}:
Because a hyperelliptic curve $C$ of genus at least $2$ cannot be embedded smoothly into $\BP^2$, we 
consider the projective closure $C'$ in $\BP^2$ of the affine curve given
by~\eqref{eq:ddz_model}.
The curve $C'$ has a unique point $\infty'$ at infinity; since this is the only singular
point on $C'$, we can identify affine points on $C$ with their images on $C'$ under the
normalization morphism $C \to C'$ and Remark~\ref{rk:prod_sing} applies.
Then the points $\infty'$ and $O = (0,0)\in C'(\Q)$ are points of maximal contact 
for the line $L_{\infty'}$ at infinity and the tangent line $L_O:y=0$, respectively, in the sense of Definition~\ref{m-contact}.
An affine Weierstrass point $P$ as in Theorem~\ref{prop:ddz}, viewed as a point on $C'$,
has the property that
the tangent line $L_P$ to $C'$ at $P$ only intersects $C'$ in $P$ and $\infty'$.
The respective intersection multiplicities are  $I_P(L_P, C') =2$ and $I_{\infty'}(L_P, C') = d-2$.
Therefore we recover the elements from Proposition~\ref{prop:ddz} by applying
Construction~\ref{constr:basic}, see Remark~\ref{rk:prod_sing}.
See Figure~\ref{fig:hyp} on page~\pageref{fig:hyp} for an example of this geometric
configuration.

Conversely, we now consider a curve $C'$ given as the projective closure of an affine curve given
by an equation~\eqref{eq:ddz_model} and a vertical line $L_\alpha$ given by $x-\alpha$,
where $\alpha \in \Q$.
This line intersects $C'$ in $\infty'$ with multiplicity $d-2$.
If it intersects $C'$ in another rational point $P$ with multiplicity~2 (in other
words, if $L_\alpha$ is tangent to $C'$ in $P$), then we
are in the situation of Construction~\ref{constr:basic}, see Remark~\ref{rk:prod_sing}.

\begin{lemma}\label{lemma:hyp_ex}
    The vertical line $L_\alpha$ is a tangent line to $C'$ in the point $P=
    (\alpha, \beta)\in C'(\Q)$ if and only if 
    \begin{equation}\label{eq:key}
      \beta^2 = \alpha^d\;\;\textrm{and}\;\; f_1(\alpha) = -2\beta.
    \end{equation}
\end{lemma}
\begin{proof}
  It is easy to see that~\eqref{eq:key} 
is equivalent to 
\begin{equation}\label{eq:hyp_cond}
    F(\alpha, y) = (y-\beta)^2.
\end{equation}
The line $L_\alpha$ intersects $C'$ only in $P$ and $\infty'$ if and only
if~\eqref{eq:hyp_cond} holds,
because the zeros of the left hand side are precisely the $y$-coordinates of the affine
intersection points between $C'$ and $L_{\alpha}$. 
\end{proof}

\begin{rk}\label{rk:connection}
  A point $P=(\alpha,\beta)$ with tangent line $L_\alpha$ as in Lemma~\ref{lemma:hyp_ex} is a rational affine Weierstrass point on the
  normalization $C$ of $C'$ and hence corresponds to a point $P$ as in
  Proposition~\ref{prop:ddz}.
\end{rk}

We can use Lemma~\ref{lemma:hyp_ex} as a recipe for~\emph{constructing} examples of
hyperelliptic curves with a number of given representatives of elements of
$K^T_2(C)/\mathrm{torsion}$.
If we fix $x_i$ and $y_i \in\Q$ such that $y_i^2=x_i^d$ for $i = 1,\ldots,m$ and treat the 
coefficients $b_0,\ldots,b_g$ of $f_1$ as indeterminates,
then the crucial condition~\eqref{eq:key} $f_1(x_i) = -2y_i$ translates into a system of $m$ linear equations. 
Generically, the set of solutions to this system is parametrized by $g+1 - m$ parameters.
Therefore the maximal number of representatives of elements of $K_2^T(C)/\mathrm{torsion}$ that we can
construct generically using this approach is $g+1$. 
This leads to a unique solution of the system of linear equations, hence to a $g+1$-parameter family of examples.

We first carry this out for the case $d=2g+1$.
\begin{lemma}\label{lem:hyp_ex_odd}
Let $g \ge 1$, $a_1,\ldots,a_{g+1} \in \Q^\times$ such that
    $a_1^2,\ldots,a_{g+1}^2$ are pairwise distinct.
    Let $V \in \mathrm{GL}_{g+1}(\Q)$ denote the
    Vandermonde matrix $\left(a_i^{2j}\right)_{\substack{1\le i \le g+1\\ 0 \le j \le g}}$
    and let $w = \left(-2a^{2g+1}_i\right)_{1\le i \le g+1} \in \Q^{g+1}$.
    Define 
    \[
        b = (b_0,\ldots,b_g) = V^{-1}\cdot w  \in \Q^{g+1}
    \]   
    and let 
    \[
        f_1(x) = b_0+b_1x+\ldots + b_gx^g.
    \]
    Then we have
\[
  f_1(a_i^2)=-2a_i^{2g+1},\quad i \in \{1,\ldots,g+1\}.
\]
\end{lemma}
\begin{proof}
  This is a consequence of classical interpolation properties of the
  Vandermonde matrix.
\end{proof}
\begin{prop}\label{prop:hyp_ex_odd}
  Let $C$ denote the hyperelliptic curve associated to the affine model
    \[
        y^2+f_1(x)y+x^{2g+1} =0,
    \]    
    where $f_1$ is as in Lemma~\ref{lem:hyp_ex_odd}.
    Then, for every $i\in\{1,\ldots,g+1\}$, we get an element $\{\infty, O, P_i\}\in
    K_2^T(C)/\mathrm{torsion}$, where $P_i = (a^2_i,a_i^{2g+1}) \in C(\Q)$.
\end{prop}
\begin{proof}
The result follows immediately from Lemma~\ref{lem:hyp_ex_odd} using Lemma~\ref{lemma:hyp_ex}.
\end{proof}
Note that curves $C$ as in Proposition~\ref{prop:hyp_ex_odd} provide
examples of the construction in~\cite{ddz} such that the two-torsion-polynomial $t(x)$ has at least
$g+2$ rational factors.
This seems to be the easiest and most natural approach when $d=2g+1$.

Let us write down the family we get from applying Proposition~\ref{prop:hyp_ex_odd} for
$g=2$.
This recovers~\cite[Example~7.3]{ddz}, see also~\cite[Remark~7.4]{ddz}.
\begin{ex}\label{ex:g2m3}
Let $a_1,\,a_2,\,a_3 \in \Q^\times$ such that $a_1^2, a_2^2, a_3^2$
are pairwise distinct and set
\begin{align*}
    \gamma =&\, (a_1+a_2)(a_1+a_3)(a_2+a_3)\\
  b_0 =& -\frac{2}{\gamma}(a_1a_2+a_1a_3+a_2a_3)a_3^2a_2^2a_1^2 \\
  b_1 =&\, \frac{2}{\gamma}(a_1^3a_2^3+a_1^3a_2^2a_3+a_1^3a_2a_3^2+a_1^3a_3^3+a_1^2a_2^3a_3+a_1
^2a_2^2a_3^2+a_1^2a_2a_3^3+a_1a_2^3a_3^2+a_1a_2^2a_3^3+a_2^3a_3^3)\\
b_2 =& -\frac{2}{\gamma}(a_1^3a_2+a_1^3a_3+a_1^2a_2^2+2a_1^2a_2a_3+a_1^
2a_3^2+a_1a_2^3+2a_1a_2^2a_3+2a_1a_2a_3^2+a_1a_3^3+a_2^3a_3
\\&+a_2^2a_3^2+a_2a_3^3).
\end{align*}
By Proposition~\ref{prop:hyp_ex_odd}, the genus~2 curve
\[
    C:y^2 + (b_0+b_1x+b_2x^2)y + x^5 = 0
\]
satisfies $\{\infty, O, P_i\} \in K_2^T(C)/\mathrm{torsion}$ for $i = 1,2,3$, where $P_i =
(a_i^2, a_i^5) \in C(\Q)$.
See Figure~\ref{fig:hyp} on page~\pageref{fig:hyp} for the example $a_1=1,a_2 = 1/2, a_3 = 1/4$.
\end{ex}

\begin{figure}
  \begin{center}
    \scalebox{.6}{\includegraphics[width=\textwidth]{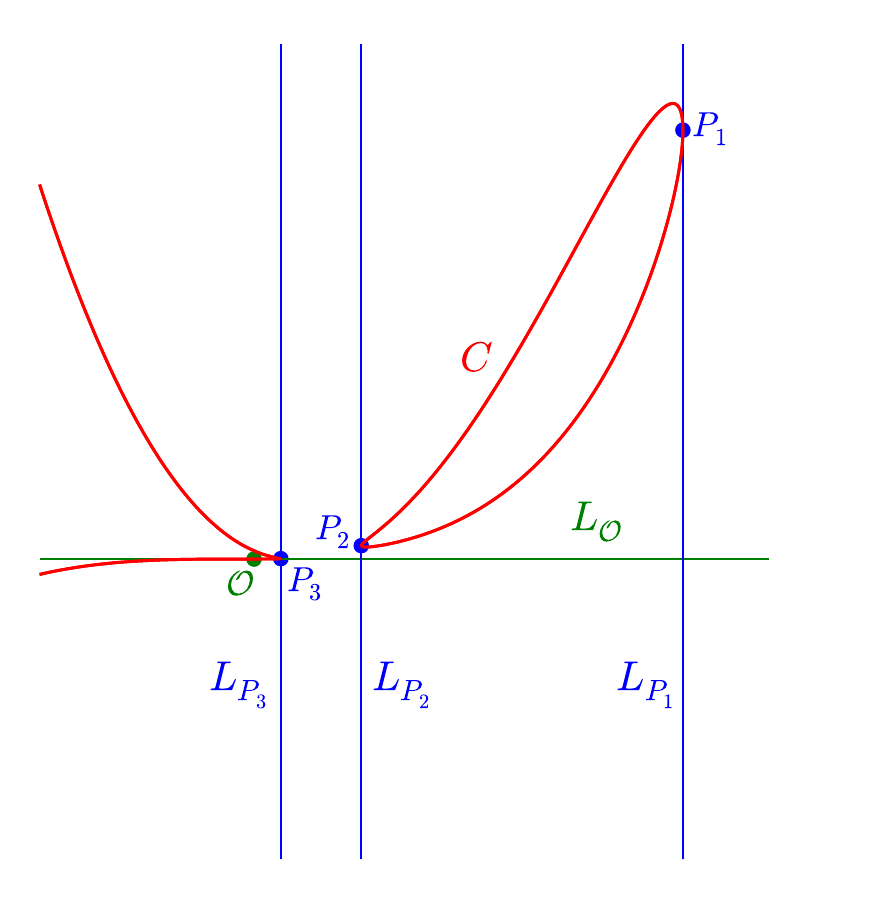}}
    \caption{The curve $C$ from Example~\ref{ex:g2m3} with $a_1=1,a_2 = 1/2, a_3 = 1/4$}
  \label{fig:hyp}
  \end{center}
\end{figure}

We could also use Lemma~\ref{lemma:hyp_ex} to find examples for $d=2g+2, m=g+1$ in a
completely analogous way.
However, here we have more flexibility. Namely,  when
$d=2g+1$ is odd, then we have
to make sure that the $x$-coordinate of the points $P=(\alpha, \beta)$ are rational squares in order
to satisfy~\eqref{eq:key}.
In contrast, when $d=2g+2$ is even,  then $P$ needs to satisfy 
\[
  \beta = \pm \alpha^{g+1},
\]
without any a priori conditions on $\alpha$.
\begin{lemma}\label{lem:hyp_ex_even}
    Let $g \ge 1$ and let $a_1,\ldots,a_{g+1} \in \Q^\times$ be pairwise distinct.
    Let $V \in \mathrm{GL}_{g+1}(\Q)$ denote the
    Vandermonde matrix $\left(a_i^j\right)_{\substack{1\le i \le g+1\\ 0 \le j \le g}}$
    and let $w = \left(-2a^{g+1}_i- 2\epsilon_i a^{g+1}_i\right)_{1\le i \le g+1} \in
    \Q^{g+1}$, where $\epsilon_1,\ldots,\epsilon_{g+1} \in \{\pm1\}$.
    Define 
    \[
        b = (b_0,\ldots,b_g) = V^{-1}\cdot w  \in \Q^{g+1}
    \]   
    and let 
    \[
        f_1(x) = b_0+b_1x+\ldots + b_gx^g+2x^{g+1}. 
    \]
    Then we have 
  \[
    f_1(a_i)=-2\epsilon_ia_i^{g+1},\quad i \in \{1,\ldots,g+1\}.
  \]
\end{lemma}
\begin{proof}
  As in Lemma~\ref{lem:hyp_ex_even}, this follows from classical interpolation properties of the
  Vandermonde matrix.
\end{proof}
\begin{prop}\label{prop:hyp_ex_even}
  Let $C$ denote the hyperelliptic curve associated to the affine model
    \[
        y^2+f_1(x)y+x^{2g+1} =0,
    \]    
    where $f_1$ is as in Lemma~\ref{lem:hyp_ex_even} and $\epsilon_1,\ldots,\epsilon_{g+1} \in \{\pm1\}$ are not all equal to~$-1$.
    Then, for every $i\in\{1,\ldots,g+1\}$, we get an element $\{\infty, O, P_i\}\in
    K_2^T(C)/\mathrm{torsion}$, where $P_i = (a_i,\epsilon_ia_i^{g+1}) \in C(\Q)$.
\end{prop}
\begin{proof}
This follows from Lemma~\ref{lem:hyp_ex_even} using Lemma~\ref{lemma:hyp_ex}.
We need that not all $\epsilon_i$ are equal to $-1$, as otherwise
$f_1$ is the zero polynomial.
\end{proof}
\begin{rk}\label{rk:univ_reln}
By~\cite[Proposition~6.14]{ddz} there is a universal relation between the
classes of the elements from Proposition~\ref{prop:hyp_ex_even} in $K_2^T(C)/\mathrm{torsion}$.
\end{rk}
Curves $C$ as in Proposition~\ref{prop:hyp_ex_even} provide
examples of the construction in~\cite{ddz} such that the two-torsion-polynomial $t(x)$ has at least
$g+2$ rational factors.
In the language of~\cite{ddz}, the indices $i$ such that $\epsilon_i=-1$ correspond to
prescribed rational roots of the first factor in the factorisation 
\begin{equation}\label{eq:t-split}
  t(x) = \frac{f_1(x)^2}{4} -x^{2g+2} = \left(\frac{f_1(x)}{2} - x^{g+1}\right)\left(\frac{f_1(x)}{2} +
  x^{g+1}\right)
\end{equation}
of $t(x)$, whereas the indices $i$ such that $\epsilon_i=1$ correspond to
prescribed rational roots of the second factor.
From this point of view it is clear that we cannot prescribe $g+1$ points $P_i$ having
$\epsilon_i=-1$, because the first factor $\frac{f_1(x)}{2} - x^{g+1}$ only has degree $g$.

Proposition~\ref{prop:hyp_ex_even} with $\epsilon_i = 1$ for $1\le i \le g+1$ recovers the ``generic'' case
of~\cite[Example~7.9]{ddz}. Moreover,~\cite[Example~10.5]{ddz} is a special case
of Proposition~\ref{prop:hyp_ex_even} when $g=2$.

As Example~\ref{ex:g2m3} suggests, the curves we obtain 
using Propositions~\ref{prop:hyp_ex_odd} and~\ref{prop:hyp_ex_even} 
may have rather unwieldy coefficients, which complicates, for instance, numerically verifying
Conjecture~\ref{conj:beilinson}. 
Instead, we can prescribe $m \le g$ vertical tangents at rational points using
Lemma~\ref{lemma:hyp_ex}.
Soling the resulting system of linear equations, the
parametrized set of solutions may yield curves whose coefficients are relatively small and
for which we know at least $m$ representatives of elements of $K^T_2(C)/\mathrm{torsion}$.
In light of Conjecture~\ref{conj:beilinson}, this is especially interesting when $m=g$.
We do not go into detail, but note that taking $g=m=2,\, d=5$ recovers
\cite[Example~7.3]{ddz} and hence \cite[Example~10.1]{ddz}. 
Taking $g=3,\,d=6, m=2, a_1 = 1 $ and $a_2 = 1/2$ recovers~\cite[Example~10.6]{ddz}.  
Furthermore, taking $g$ arbitrary, $d=2g+2$, $m=g$ and $\epsilon_i =-1$ for $1\le i\le g$ recovers~\cite[Example~10.8]{ddz}.
This example is especially important, because de Jeu~\cite{deJeu:limit} showed that for
this family one gets
$g$~\emph{independent} elements of $K^T_2(C)/\mathrm{torsion}$ using a limit formula for the regulator of
these elements, see also~\cite[\S6]{dl}.

\begin{rk}\label{rk:mixed}
  Note that in all examples for $d=2g+2$ considered in~\cite{ddz}, the prescribed roots of
  the~2-torsion polynomial $t(x)$ were either all roots of the first
  factor (corresponding to $\epsilon_i=-1$)
  or all roots of the second factor (corresponding to $\epsilon_i=-1$) in the
  factorisation~\eqref{eq:t-split}.
  Therefore we can easily construct examples which do not appear in~\cite{ddz}, by
  considering a ``mixed'' situation, where not all $\epsilon_i$ are the same.
  This observation is due to Rob de Jeu.
\end{rk}

\begin{ex}\label{ex:mixed_g2}
  We apply Proposition~\ref{prop:hyp_ex_even} with $g=2,\,\epsilon_1=1$ and
  $\epsilon_2=\epsilon_3 = -1$. 
  Setting 
  \[
    \gamma = a_1^2-a_1a_2-a_1a_3+a_2a_3,
  \]
  this leads to the following family, where $a_1, a_2, a_3 \in \Q^\times$ are pairwise distinct:
\[
C:y^2 +
\left(-\frac{4a_1^3a_2a_3}{\gamma}+\frac{4a_1^3(a_2+a_3)}{\gamma}x-\frac{4a_1^3}{\gamma}x^2+2x^3\right)y +
x^6 = 0
\]
Here we have~3 elements $\{\infty,O, P_i\}\in K_2^T(C)/\mathrm{torsion}$, where $P_1 = (a_1,a_1^3),\, P_2 =
(a_2,-a_2^3),\, P_3 = (a_3, -a_3^3) \in C(\Q)$.
\end{ex}

\begin{rk}\label{rk:2inf}
  So far we have only considered hyperelliptic curves with a unique point at infinity. 
  The main reason is that in this situation Remark~\ref{rk:prod_sing} applies, so we can
  use Construction~\ref{constr:torsion}.
  In the situation where $O$ is a point of maximal contact, but there are two points at
  infinity lying above the singular point at infinity on $C'$, the construction discussed in the present
  section breaks down.
  In Section~\ref{sec:nekovar} we will show how to construct examples of hyperelliptic
  curves with two points at infinity with a given element of $K_2^T(C)/\mathrm{torsion}$
  using a different method, see Example~\ref{ex:hyp_nek}.
\end{rk}

\subsection{Integrality}\label{sec:hyp_integrality}
Lemma~\ref{lemma:hyp_ex} provides a method for constructing hyperelliptic curves $C/\Q$
with given elements of $K^T_2(C)/\mathrm{torsion}$ as in Propositions~\ref{prop:hyp_ex_odd}
and~\ref{prop:hyp_ex_even}.
Provided a simple condition is satisfied, it is easy to deduce from  
the results of~\cite{ddz} that these elements are actually integral. 
Recall that we call an element of $K_2^T(C)/\mathrm{torsion}$ integral if it
lies in $K_2(C;\Z)$,
the subgroup of $K^T_2(C)/\mathrm{torsion}$ that appears in the formulation
of Beilinson's Conjecture~\ref{conj:beilinson}.
Dokchitser, de Jeu and Zagier prove the following result:
\begin{thm}\label{thm:ddz_integral}
    Let $g \ge 1$, let $d \in \{2g+1, 2g+2\}$ and let $f_1$ be as in~\eqref{eq:f1} such that $f_1$ has integral coefficients. 
    Consider a rational root $\alpha$ of $t(x) = -x^d+f_1(x)^2/4$ and let $P = (\alpha,
    -f_1(\alpha)/2) \in C(\Q)$, where $C$ is the hyperelliptic curve given by the affine
    equation $y^2+f_1(x)y = x^d$. 
    Then, if $1/\alpha \in \Z$, the element $2\{\infty, O, P\}$ is integral.
\end{thm}
\begin{proof}
    This is a special case of~\cite[Theorem~8.3]{ddz}.
\end{proof}
\begin{cor}
  In the situation of Proposition~\ref{prop:hyp_ex_odd} or
  Proposition~\ref{prop:hyp_ex_even}, let $1/a_i \in\Z$.
  Then the element $2\{\infty, O, P_i\}$ is integral.
\end{cor}
\begin{proof}
    Recall that the points $P_i=(x_i,y_i)$ are precisely the rational affine Weierstrass points $P$
which appear in Proposition~\ref{prop:ddz}.
    So if $f_1 \in \Z[x]$ and $1/x_i \in \Z$, then Theorem~\ref{thm:ddz_integral} implies that 
  $2\{\infty, O, P_i\}$ is integral.

If $f_1 \notin \Z[x]$, then we can find a polynomial $\widetilde{f_1}\in \Z[x]$ and a rational number $a \in
\Q^\times$ such that the given model of $C$ can be transformed into 
\[
    \widetilde{C} : y^2 + \widetilde{f_1}(x)y + x^d =0
\]
via the transformation $\psi:C\to\widetilde{C}$ taking
$(x,y) \in C$ to $(x, ay)\in \widetilde{C}$.
Hence the reciprocal of the $x$-coordinate of $\psi(P_i)$ is
equal to $1/x_i$.
So if $1/x_i \in \Z$, then 
\[
2\psi(\{{\infty}, {O}, {P_i}\}) =   2\{{\psi(\infty)}, {\psi(O)}, {\psi(P_i)}\}
\]
is integral by Theorem~\ref{thm:ddz_integral}, and therefore 
  $2\{\infty, {O}, {P_i}\}$ is integral as well.
\end{proof}
In particular, for each $m \le g+1$, we can find families of hyperelliptic curves in $m$
integral parameters  having $m$ explicitly given elements of $K_2(C;\Z)$.
Unfortunately, it is in general not at all easy to check how many of these elements are
actually independent, but see~\cite{dl}.

\section{Elements on quartics coming from line configurations}\label{sec:spq}
We saw in the previous section that by results of~\cite{ddz} and~\cite{deJeu:limit}  one
can construct hyperelliptic curves $C$ of arbitrary
genus $g$ with at least $g$ independent elements of $K^T_2(C)/\mathrm{torsion}$ (and even $K_2(C;\Z)$) using vertical tangent
lines and lines having maximal contact multiplicity with $C$.
We would like to adapt this strategy to non-hyperelliptic curves.
Because all genus~2 curves are hyperelliptic, it is natural to consider smooth plane
quartics, since every non-hyperelliptic curve of genus~3 can be
canonically embedded into $\BP^2$ as a smooth plane quartic.

Similar to Section~\ref{sec:hyp}, we will use smooth projective plane curves $C$ of degree~4 having
two rational points $\infty$ and $O$ of maximal contact with lines $L_{\infty}$ and
$L_{O}$, respectively.
Such points are usually called {\em hyperflexes}.
The difference is that in the quartic situation, such a curve can be embedded as a smooth
curve into $\BP^2$ (without weights), so that we can work directly with the smooth curve
$C$ and the lines are simply the respective tangent lines to $C$.
Using a transformation, if necessary, we might as well assume that $\infty=(0:1:0)$ is the unique point at
infinity on $C$ and that $O = (0:0:1)$. 
Vermeulen~\cite{vermeulen} showed that in this case $C$ can be described as the
projective closure in $\BP^2$ of the affine curve given by an equation
\begin{equation}\label{eq:spq}
    F(x,y) :=  y^3 + f_{2}(x)y^{2} + f_1(x)y + x^4 = 0,
\end{equation}
where $\deg(f_1) \le 2$ and $\deg(f_2) \le 1$.
For the remainder of the present section, we assume that $C$ is of this form.
Then the divisor $(O) -(\infty)$ is a torsion divisor induced by the tangent line $L_O:
y=0$, since
\[
  \div(y) = 4(O) -4(\infty).
\]

For our purposes, the smooth plane quartic analogue of a rational point on a hyperelliptic curve with a vertical
tangent is an inflection point having a vertical tangent $L_\alpha = x-\alpha$, where $\alpha
\in \Q$.
Such a point $P \in C$ has the property that $L_\alpha$ intersects $C$ in $P$ with multiplicity~3 and in
$\infty$ with multiplicity~1.
If such a point $P$ exist, then we are in the situation of
Construction~\ref{constr:torsion} and hence we get an element
\[
  \{\infty, O, P\} \in K_2^T(C)/\mathrm{torsion}.
\]
We first translate this geometric condition on the point $P$ (or rather the line
$L_\alpha$) into a condition on the polynomials $f_1$ and $f_2$ 

\begin{lemma}\label{lemma:quartic_vertical}
    The vertical line $L_\alpha$ intersects $C$ in the point $P=
    (\alpha, \beta)\in C(\Q)$ with multiplicity~3 if and only if 
    \begin{equation}\label{eq:spq_cond2}
      \beta^3 = -\alpha^4,\;\; f_1(\alpha) =
      3\beta^2\;\;\textrm{and}\;\;f_2(\alpha) = -3\beta .
    \end{equation}

\end{lemma}
\begin{proof}
  In analogy with the proof of Lemma~\ref{lemma:hyp_ex}, 
  the conditions~\eqref{eq:spq_cond2} are equivalent to 
\begin{equation}\label{eq:spq_cond}
    F(\alpha, y) = (y-\beta)^3.
\end{equation}
The assertion follows,
because the roots of $F(\alpha,y)$ are the $y$-coordinates of the affine
intersection points between $C$ and $L_{\alpha}$.
\end{proof}

For a positive integer $m$, we can try to use Lemma~\ref{lemma:quartic_vertical} to construct smooth plane quartics $C:F=0$,
where $F$ is as in~\eqref{eq:spq},  which have $m$ inflection points $P_i$ with vertical tangent
lines, giving rise to $m$ elements $\{\infty, O, P_i\} \in K_2^T(C)/\mathrm{torsion}$ by
virtue of Construction~\ref{constr:torsion}, similar to Proposition~\ref{prop:hyp_ex_odd}
and~\ref{prop:hyp_ex_even}.
This amounts to solving a system of linear equations, given by~\eqref{eq:spq_cond2}.
But we can only expect polynomials $f_1$ and $f_2$ as in the Lemma to exist if 
$m \le 2$, because $\deg(f_2) \le 1$.
The case $m=2$ would lead to a smooth plane quartic $C$ for which we have elements
\[
\{\infty, O, P_1\}, \{\infty, {O}, {P_2}\}, \{{\infty}, {P_1},
{P_2}\} \in K^T_2(C)/\mathrm{torsion}.
\]
However, by a calculation analogous to~\cite[Example~5.2]{ddz}, 
the element $\{{\infty}, {P_1}, {P_2}\}$ is trivial in $K^T_2(C)/\mathrm{torsion}$,
 so that for $m\le 2$ we only get at most~2 independent elements of
 $K^T_2(C)/\mathrm{torsion}$ because of parts (iii), (iv) and (v) of Proposition~\ref{prop:constr_props}.

\begin{rk}
    Suppose that $F$ is of the form \eqref{eq:spq} and that $C:F=0$ is the corresponding 
    smooth plane quartic and such that there are~3 distinct inflection points
    $P_i = (x_i,y_i) \in C(\Q)$ having vertical tangents.
Then it turns out that such a curve corresponds
to a nonsingular $\Q$-rational point $(a_1,a_2,a_3)$ on the projective curve 
defined in $\BP^2$ by
\[
 a_1^2a_2^2+a_1^2a_2a_3+a_1^2a_3^2+a_1a_2^2a_3+a_1a_2a_3^2+
 a_2^2a_3^2=0,
\]
where $x_i = a_i^3$ and $y_i = -a_i^4$.
Using (for instance) the computer algebra system {\tt Magma}~\cite{magma}, one sees easily
that this curve has~3 singular points
$(0:0:1), (0:1:0)$ and $(1:0:0)$ and 
its normalization is isomorphic over $\Q$ to the conic defined by
\[
X^2 + XY + Y^2 + XZ + YZ + Z^2=0\,,
\]
which has no rational points.
Hence, no such curve $C$ can exist.
\end{rk}

Therefore inflection points with vertical tangent lines are not sufficient to construct~3
independent elements of $K^T_2(C)/\mathrm{torsion}$ for a smooth plane quartic $C$ with two hyperflex points in
$\infty$ and $O$.
Instead, we will only use one such inflection point.
In addition, we will construct an inflection point whose tangent line also intersects the quartic
in $O$ (instead of
$\infty$).
To this end, we work on a different model $D$, where the parts of $O$ and $\infty$ are
interchanged and use an inflection point on $D$ with a vertical tangent.
These conditions lead to a system of linear equations whose solution gives us the family
presented in Theorem~\ref{thm:spq_fam1} below.
See Figure~\ref{fig:qex1} on page~\pageref{fig:qex1} for an example of this geometric
configuration.

Recall that the {\em discriminant} of a projective plane quartic curve is defined
to be the discriminant of the ternary quartic form defining the curve, cf.~\cite{salmon}.
It vanishes if and only if the curve is singular.
\begin{thm}\label{thm:spq_fam1}
    Let $a,b\in \Q^\times$ such that $a \ne b$, let $c \in \Q$ and let 
    \begin{equation}\label{quartic-param}
        f_1(x) = a^6b^6+a^3b^3cx+(3a^2-b^6-b^3c)x^2\;\textrm{ and }\;
    f_2(x) = 3a^4-a^3c+cx.
\end{equation}
Suppose that $\mathrm{disc}(C)\ne0$.
Then the smooth plane quartic $C$ given by the affine equation
\[
    F(x,y) =  y^3 + f_{2}(x)y^{2} + f_1(x)y + x^4 = 0
\]    
has the following properties:
\begin{enumerate}[(i)]
    \item We have $P = (a^3, -a^4) \in C(\Q)$ and the tangent line 
    $ L_P:\, x=a^3 $
    through $P$ has contact multiplicity $3$ with $C$; the other point of  intersection  is $\infty$. 
    \item We have $Q = (-a^2b^3, -a^2b^6) \in C(\Q)$ and the tangent line $L_Q:\, y= b^3 x$ 
    through $Q$ has contact multiplicity $3$ with $C$; the other point of
    intersection  is $O$. 
  \item We get the following elements of $K^T_2(C)/\mathrm{torsion}$:
    \[
    \{{\infty}, O, P\},\,\{{\infty}, O, Q\},\, \{{\infty}, P, Q\}
  \]
\end{enumerate}
\end{thm}
\begin{proof}
We have 
\[
  (a^3)^4 = -(-a^4)^3,\;\; f_1(a^3) = 3a^8\;\;\textrm{and}\;\; f_2(a^3) = 3a^4,
\]
so that (i) follows immediately from Lemma~\ref{lemma:quartic_vertical}.

To prove (ii), we again want to use Lemma~\ref{lemma:quartic_vertical}.
  To this end, we apply the projective transformation
\[
    \varphi: C \to D,\quad (X:Y:Z) \mapsto (X:a^2b^2Z:Y),
\]
where $D$ is the smooth plane quartic given by the affine equation
\[
    D: G(w, z) =z^3 + g_2(w)z^2 + g_1(w)z + w^4 = 0.
\]
This transformation maps $O$ to the unique point at infinity $(0:1:0) \in D$, it maps $\infty$ to
$(0,0) \in D$ and $Q$ to $\varphi(Q) = (\frac{1}{b^3}, -\frac{1}{b^4})$.
Because we also have
\[
  g_1\left(\frac{1}{b^3}\right) =  \frac{3}{b^8} \;\;\textrm{and}\;\;
    g_2\left(\frac{1}{b^3}\right) = -\frac{3}{b^4},
\]  
Lemma~\ref{lemma:quartic_vertical} is applicable exactly as in (i) and assertion (ii) follows.

It remains to prove (iii).
For this it suffices to note that by (i) we have
\[
    \div(x-a^3) = 3(P) - 3(\infty)
\]
and by (ii) we have
\[
    \div\left(\frac{y-b^3x}{y}\right) = 3(Q)-3(O),
\]
where we consider the functions on the left hand sides as functions on $C$.
Hence the pairwise differences $(R) - (S)$  are torsion divisors for all $R, S$ in the set
$\{\infty, O, P, Q\}$ and the result follows from Proposition~\ref{prop:constr_props}.
\end{proof}

\begin{rk}\label{rk:symm}
    If $c=0$ in Theorem~\ref{thm:spq_fam1}, then we have the additional rational points
    $P' = (-a^3, -a^4)$ and $ Q'= (a^2b^3, -a^2b^6)$ on $C$. 
    The respective tangent lines $L_{P'} :\, x=-a^3$ and $L_{Q'} :\, y = -b^3x$  through $P'$
    and $Q'$ also have  contact multiplicity $3$ with $C$ at these points and intersect
    $C$ in $\infty$ and $O$, respectively. 
    This yields
the following additional   elements of $K^T_2(C)/\mathrm{torsion}$:
\[
  \{\infty, O, P'\}, \{\infty, O, Q'\}, \{\infty, P', Q'\},
  \{\infty, P, Q'\}, \{\infty, P', Q\}, \{\O, P', Q\}, \{\infty, P, Q'\}
\]
 \end{rk}

\begin{rk}\label{rk:disc}
    Invariants of non-hyperelliptic curves of genus~3 can be computed, for instance, using
    David Kohel's {\tt Magma} package {\tt Echidna} \cite{echidna}.
    With its aid, we find that the discriminant of the curve $C$ from Theorem~\ref{thm:spq_fam1}
    is equal to
    \[
        \textrm{disc}(C) =
a^{42}b^{36}(a + b^3)^3(2a - 2b^3 - c)^6(3a - 2b^3 - c)(3a -
b^3 - c)(6a + 2b^3 + c)^2q(a,b,c),
\]
where 
\begin{align*}
  q(a,b,c)=&\;9a^5 + 30a^4b^3 - 24a^4c +
4a^3b^6 + 23a^3b^3c + 16a^3c^2 + 12a^2b^9 +
3a^2b^6c + 12a^2b^3c^2 \\
&- 3ab^{12} - 3ab^9c - 2b^{15} - 5b^{12}c - 4b^9c^2 - b^6c^3,
 \end{align*}
 making it easy to check whether $C$ is smooth or not.
\end{rk}

\begin{rk}
  It is tempting to generalize the approach described above to non-hyperelliptic curves of
  genus $g>3$ which have two points of maximal contact with
  lines and have tangent lines which pass through one of these two points and precisely one
  other point, 
  However, we have found that this does not allow us to construct $g$ elements of
  $K^T_2(C)/\mathrm{torsion}$, since the dimensions of the solution spaces of the resulting systems of
  linear equations are too small.
\end{rk}

\subsection{Integrality}\label{sec:spq_integrality}
Now we investigate when the elements of $K^T_2(C)/\mathrm{torsion}$ constructed in
Theorem~\ref{thm:spq_fam1} are integral, 
It turns out that we can easily write down a family in~three integral parameters such that
two of the elements are integral, but if we want all~three of them to be integral, then we
have to restrict to a one parameter subfamily.
\begin{prop}\label{prop:spq_fam1_int}
    Let $a,b\in \Q^\times$ such that $a \ne b$, let $c \in \Q$ and let 
    $C$ be as in Theorem~\ref{thm:spq_fam1} such that $\mathrm{disc}(C)\ne0$.
    \begin{enumerate}[(i)]
        \item If $1/a , b, c\in \Z$, then 
             $\{\infty, O, P\}$ and $ \{\infty, O,
           Q\}$ are integral.
      \item 
          If  $a=\pm\frac{1}{2}, b=\mp 1, c \in \Z$, then 
          $2\{\infty, P, Q\}$ is integral. 
    \end{enumerate}
\end{prop}
\begin{proof}
        To prove part (i), suppose that $d := 1/a, b, c \in \Z$.
Then the smooth plane quartic given by the affine model
\[
  \widetilde{C} : d^{18}y^3 + d^{12}{f_2}(x)y^2 + d^{6}f_1(x)y + x^4
\]
has coefficients in $\Z$ and is isomorphic to $C$ via the transformation
\[
  \psi : C \to \widetilde{C};\qquad (x,y) \mapsto (x, a^6y).
\]
By functoriality, the element $\{\infty, O, P\}$ induces an element
of $K_2(C;\Z)$ if and only if its image under $\psi$ induces an element of
$K_2(\widetilde{C};\Z)$.

Applying Construction~\ref{constr:torsion} to the points $\infty, O$ and $P$, we find that
\[
  \{\infty, O, P\} = \left\{-\frac{y}{a^4},\,1-\frac{x}{a^3}\right\}.
\]
Hence we get
\[
  S_1 := \psi(\{\infty, O, P\}) = \{{\psi(\infty)}, {\psi(O)}, {\psi(P)}\} =
  \left\{ h_1,h_2\right\} \in K^T_2(\widetilde{C})/\mathrm{torsion}, \\
\]
where $h_1 = -d^{10}y$ and $h_2=1-d^3x$.

We will show integrality of this element 
using a strong desingularization $\widetilde{\KC}$ (cf.  \cite[\S8.3.4]{liu}) of the
Zariski closure $\widetilde{C}^Z$ of $\widetilde{C}$ in $\BP^2_{\Z}$. 
In other words, $\widetilde{\KC}$  is a proper regular model of $\widetilde{C}$ such that
there exists a proper birational morphism 
\[\pi : \widetilde{\KC} \to   \widetilde{C}^Z\] 
which is an isomorphism outside the singular locus of $\widetilde{C}^Z$. 
Following the proof of~\cite[Theorem~8.3]{ddz}, we will prove the integrality of the class
of $S_1$ using the behaviour of the functions  $h_1$ and $h_2$ on $\widetilde{C}^Z$.

If $p$ is a prime such that $\ord_p(d) >0$, then $h_2=1$ on
$\widetilde{\KC}_p$ and hence $T_\Gamma(S_1) = 1$ for every irreducible component
$\Gamma $ of $\widetilde{\KC}_p$.

Hence we may assume that $\ord_p(d) = 0$.
If $\Gamma $ is an irreducible component of $\widetilde{\KC}_p$ such that $\pi(\Gamma) =
\Delta$ is an irreducible component of  $\widetilde{C}^Z_p$, then it is easy to see from
the given equation of $\widetilde{C}$ that
\[
  \ord_\Gamma(h_1) = \ord_\Gamma(h_2) = 0,
\]
which implies that $T_\Gamma(S_1)=1$.

In order to finish the proof that $S_1$ induces an element of
$K_2(\widetilde{C};\Z)$, we have to consider the case of a prime number $p$
such that $\ord_p(d) =0$ and such that $\pi(\Gamma) = P_0$ is a singular point of
$\widetilde{C}^Z$.
One checks directly that $P_0$ must be an affine point of $\widetilde{C}^Z_p$, say $P_0 =
(x_0, y_0)$.
Now we apply a case distinction: 
If $x_0 = 0$, then $h_2(P_0) = 1$ follows, implying that $T_\Gamma(S_1) =1$.
If, on the other hand, $x_0 \ne 0$, then we can only have $\ord_\Gamma(h_2) \ne 0$ if $x_0
= a^3$, implying $y_0 = -a^{10}$ and hence $h_1(P_0) = 1$ and $T_\Gamma(S_1)=1$.
This proves that the class of $S_1$ in $K_2^T(C)/\mathrm{torsion}$ is integral.

Next we show that the 
element $\{\infty, O, Q\}$ 
is integral.
We work on the curve $\widetilde{D}$ which is defined as the image of $\widetilde{C}$ under
the transformation 
\[\widetilde{\varphi}(x,y) = (w,z) = \left(\frac{x}{y},\frac{1}{y}\right).\]
There is a commutative diagram
\[
\xymatrix{
  C \ar[r]^\varphi \ar[d]^{\psi} & D\ar[d]\\
  \widetilde{C} \ar[r]^{\widetilde{\varphi}} & \widetilde{D},\\
}
\]
where $D$ is as in the proof of Theorem~\ref{thm:spq_fam1} and the vertical morphism on the right maps $(w,z)$ to $(w,a^2z)$.
Therefore it follows that
\[
    \widetilde{\varphi}(\psi(\{\infty, O, Q\})) = \{-b^6z, 1-b^3w\}
\]
is integral using an argument analogous to the one employed above for $S_1$.
This finishes the proof of (i).

Now we move on to a proof of (ii).
Suppose that $d=1/a, b, c \in\Z$.
As in the proof of (i), we work on the curve $\widetilde{C}$ and we compute
\[
    S_2 := \psi(\{\infty, P, Q\}) = \{h_3, h_4\},
\]
where 
\[
  h_3 = -\frac{a^2(b^3+a)}{x-a^3},\quad h_4 = -\frac{(y-a^6b^3x)^4}{a^{26}(b^3+a)^4y}.
\]
If $p$ is a prime such that $\ord_p(a) <0$, then a simple calculation
shows that $h_3=1$ on
$\widetilde{\KC}_p$ and hence $T_\Gamma(S_2) = 1$ for every irreducible component
$\Gamma $ of $\widetilde{\KC}_p$.
Hence $S_2$ is integral for $p=2$ when $a=\pm\frac{1}{2}$ and $b=\mp 1$.

Suppose that $p$ is a prime such that $\ord_p(a) = 0$ and 
$\ord_p(b^3+a)=0$. 
Note that the second condition is satisfied for every prime 
$p\ne 2$
when $a=\pm\frac{1}{2}$ and $b=\mp 1$.
If $\Gamma$ is an irreducible component 
of $\widetilde{\KC}_p$ such that $\pi(\Gamma) =
\Delta$ is an irreducible component of  $\widetilde{C}^Z_p$, then the conditions on $p$
imply $\ord_\Gamma(h_3) = \ord_\Gamma(h_4) = 0$ and thus $T_\Gamma(S_2)=1$.

It remains to consider the situation where $\Gamma$ is an irreducible component 
of $\widetilde{\KC}_p$ such that $\pi(\Gamma) = P_0$ is a singular point of
$\widetilde{C}^Z_p$.
Such a point $P_0$ must be an affine point because $\ord_p(a) = 0$, say $P_0 = (x_0,y_0)$.
We distinguish cases as follows:

If $x_0 = 0$, then we have 
\[
    h_3(P_0) = \frac{b^3+a}{a} \quad h_4(P_0) = -\frac{y_0^3}{a^{26}(b^3+a)^4}. 
\]
Hence we deduce $\ord_\Gamma(h_3)=0$ and, if $y_0\ne0$, also $\ord_\Gamma(h_4) = 0$.
If, on the other hand, $y_0=0$, then we find $\ord_\Gamma(h_4) > 0$ which could
potentially cause problems. 
However, if $a=\pm\frac{1}{2}$ and $b = \mp 1$, then we actually have $\frac{b^3+a}{a} =
-1$, so that $T_\Gamma(2S_2)=1$.

If $x_0 \ne 0$, then we must also have $y_0 \ne 0$, due to the defining equation of
$\widetilde{C}$.
This final case is easy because of our knowledge of the zeros and
poles of $h_4$ and $h_3$ on $\widetilde{C}$: 
Namely, if $\ord_\Gamma(h_4) \ne 0$, then $P_0$ is the reduction of $\psi(Q)$ and therefore, by construction,
$h_3(P_0) =1$. 
Similarly, $P_0$ must be the reduction of $\psi(P)$ whenever $\ord_\Gamma(h_3) \ne 0$; hence $h_4(P_0)=1$.

We conclude that if $a= \pm\frac{1}{2}$ and $b=\mp1$, then $T_\Gamma(2S_2)$ is trivial
for all irreducible components.
The integrality of $2\{\infty, P, Q\}$ follows.
\end{proof}
\begin{rk}\label{rk:symm_int}
    If $c=0$ in Theorem~\ref{thm:spq_fam1}, then by Remark~\ref{rk:symm}, we have the additional points $P' = (-a^3, -a^4),\; Q'= (a^2b^3, -a^2b^6) \in C(\Q)$ 
    and the following additional elements of $K^T_2(C)/\mathrm{torsion}$:
\[
  \{\infty, O, {P'}\}, \{\infty, O, {Q'}\}, \{\infty, {P'}, {Q'}\},
  \{\infty, P, {Q'}\}, \{\infty, {P'}, Q\}
\]
If $1/a, b \in \Z$, then at least the first two of these are integral. 
 \end{rk}

Using Theorem~\ref{thm:spq_fam1} and Proposition~\ref{prop:spq_fam1_int} we can write down a one-parameter family of plane
quartics $C_t$ such that any smooth $C_t$ with $t \in \Z$ has at least~3  
representatives of elements of $K_2(C;\Z)$ which we can
describe explicitly. 
\begin{cor}\label{cor:quartic_family}
    Consider the family of plane quartics defined by
    \[
       C_t : y^3 + txy^2   + (1/8t + 3/16)y^2 - (t + 1/4)x^2y   - 1/8txy + 1/64y + x^4 =0
    \]
    where $t \in \Q$.
    Then we have the following properties:
    \begin{enumerate}[(i)]
        \item The curve $C_t$ is smooth unless $t \in \{1, -3, -5/2, -7/2\}$.
        \item If $t_1, t_2 \in \Z$ are distinct, then $C_{t_1}$ and $C_{t_2}$ are not isomorphic.
        \item If $t \in \Z\setminus\{1,-3\}$, then the following elements of
          $K^T_2(C_t)/\mathrm{torsion}$ are integral: 
                \[
                    \{\infty, O, P\}, \{\infty, O, Q\}, 2\{\infty, P, Q\},
                \]
                where $P = (-\frac{1}{8}, -\frac{1}{16}),\; Q= (-\frac{1}{4},
              -\frac{1}{4}) \in C_t(\Q)$.\\ 
            If $t \in \Z\setminus\{1,-3\}$, then by means of the tangent lines 
              $L_P:\, x=-1/8 $ and $L_Q:\,y=x $ at the points  $P = (-\frac{1}{8}, -\frac{1}{16}),\; Q= (-\frac{1}{4},
              -\frac{1}{4}) \in C_t(\Q)$ we get the following integral elements of   $K^T_2(C_t)/\mathrm{torsion}$:
              \[
                    \{\infty, O, P\}, \{\infty, O, Q\}, 2\{\infty, P, Q\},
                \]
    \end{enumerate}
\end{cor}
\begin{proof}
The discriminant of $C_t$ factors as
\[
    2^{-52}(t-1)^2(t+3)^6(2t+5)(2t+7)(32t^3 + 96t^2 - 12t + 5),
\]
proving (i).

To prove (ii) it suffices to compute the first Dixmier-Ohno invariant (see~\cite{dixmier}) $I_3$ of $C_t$ which
is equal to $I_3 = 2^{-7}(16t^3 + 84t^2 + 98t - 15)$ and vanishes if and only if $C_t$ is
singular.
Elementary arguments show that no two distinct integers can lead to the same $I_3$.

The equation for $C_t$ is derived from the equation of the curve $C$ in
the  Theorem~\ref{thm:spq_fam1} by specialising to $a=-\frac{1}{2}$ and $b=1$ and
replacing $c$ by $t$.
Hence (iii) follows immediately from Proposition~\ref{prop:spq_fam1_int}.
\end{proof}
See Figure~\ref{fig:qex1} on page~\pageref{fig:qex1} for the case $t = 0$.

\begin{figure}
  \begin{center}
    \scalebox{.6}{\includegraphics[width=\textwidth]{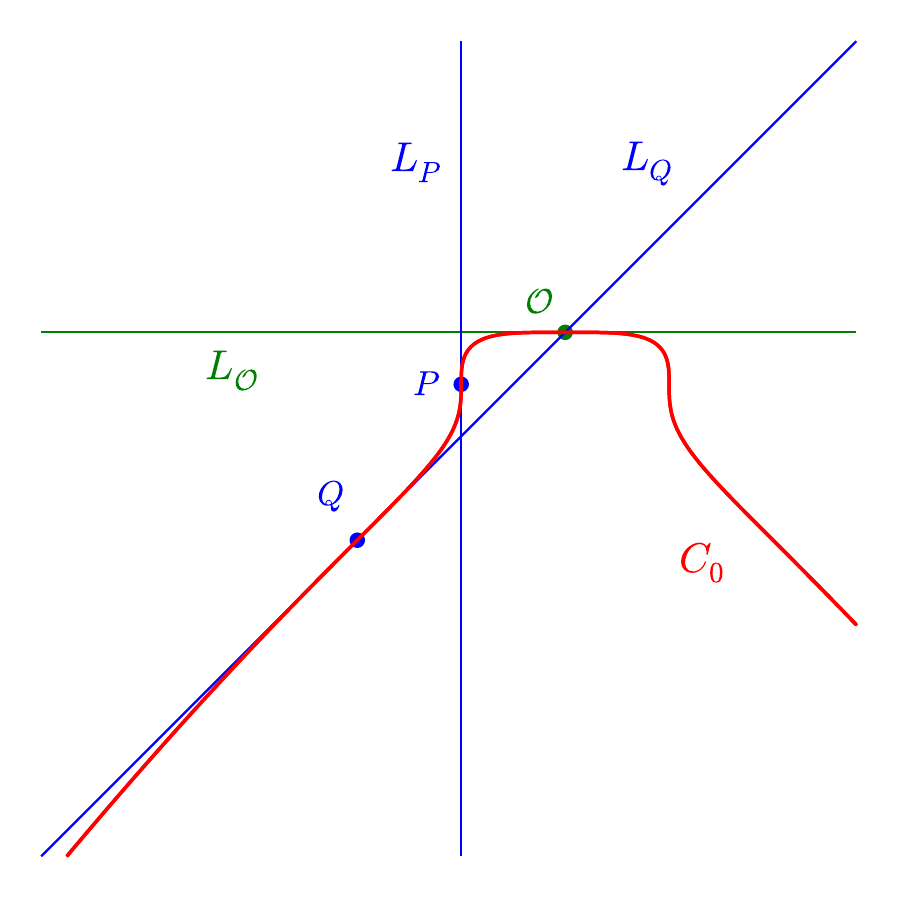}}
    \caption{The curve $C_0$ from Corollary~\ref{cor:quartic_family}}
  \label{fig:qex1}
  \end{center}
\end{figure}

\begin{rk}
    Instead of choosing $a=-\frac{1}{2}$ and $b=1$ in~\eqref{quartic-param}, we could have chosen the family of
    curves $C'_t$ obtained by setting $a=\frac{1}{2}$ and $b=-1$. 
    But if $t \in \Q$, then we have $C_t \cong C'_{-t}$ via $(x,y) \mapsto (-x,y)$.
\end{rk}

\section{Elements on quartics coming from conics and lines}\label{sec:conics}
If we want to apply Construction~\ref{constr:torsion}, then it is not
necessary to restrict to the situation where the elements in $K_2^T$ are constructed using
lines.
Instead it is also possible to use curves of higher degree
and we give an example of such a construction: We find a family of smooth
plane quartics over $\Q$ that have elements in $K_2^T$ which are constructed using lines and
conics.
See Figure~\ref{fig:qex2} on page~\pageref{fig:qex2}.

As in Section~\ref{sec:spq}, we work on smooth plane quartics with two hyperflex points
$\infty = (0:1:0)$ and $O = (0,0)$ with respective tangent lines $L_\infty$ (the line at
infinity) and $L_O : y=0 $, so
we may assume that they are given as in \eqref{eq:spq}, namely by an equation
\begin{equation}\label{eq:fcq}
  F(x,y) =  y^3 + f_{2}(x)y^{2} + f_1(x)y + x^4 = 0,
\end{equation}
where $f_1,\,f_2 \in \Q[x]$, $\deg(f_1) \le 2$ and $\deg(f_2) \le 1$.

Suppose that $D/\Q$ is a projective conic, defined by the affine equation 
\begin{equation}\label{conic}
  (y+d_1x + d_2)^2 + d_3x + d_4x^2=0,
\end{equation}  
where $d_1,\ldots,d_4 \in \Q$.
Then we can construct an equation~\eqref{eq:fcq} of a plane quartic $C$ having maximal contact with $D$.
\begin{lemma}\label{lemma:qc}
  Let $d_1,\ldots,d_4 \in \Q$ and let
  \begin{equation}\label{quarticconic}
    F(x,y) = ((y+d_1x + d_2)^2 + d_3x + d_4x^2)y + x^4.
  \end{equation}
  Then the plane quartic $C$ defined by $F(x,y)=0$ has contact
  multiplicity~8 with the conic $D$ at the point $R = (0,-d_2)$.
  Moreover, if $C$ is smooth, then we have an element 
  \[
    \{\infty, O, R\} \in K_2^T(C)/\mathrm{torsion}.
  \]
\end{lemma}
\begin{proof}
  Since we can add a multiple of the defining equation of $D$ to the defining equation of
  $C$ without changing the intersection multiplicity $I_R(C,D)$ between $C$ and $D$ at $R$,
  the latter is equal to $4I_R(L, D)$, where $L$ is the $y$-axis.
  Hence $I_R(C,D) = 8$. 
  This means that the function
  \[
  h = (y+d_1x + d_2)^2 + d_3x + d_4x^2 \in K(C)^\times
  \]
  satisfies 
  \[
    \div(h) = 8(R) - 8(\infty).
  \]
  Because we have 
  \[
    \div(y) = 4(O) - 4(\infty)
  \]
  as before, Construction~\ref{constr:torsion} is applicable and the claim follows.
\end{proof}
So if $C$ is as in Lemma~\ref{lemma:qc}, and $C$ is smooth, then we know one representative of an element of
$K_2^T(C)/\mathrm{torsion}$.
Smoothness can be checked by computing the discriminant of $C$ using {\tt
Echidna}~\cite{echidna}; it turns out that, in particular, $C$ is singular when $d_2=0$ or
$d_3=0$.
In order to construct further elements, we can combine Lemma~\ref{lemma:qc} with the approaches
discussed in Section~\ref{sec:spq}.
Namely, we can choose the coefficients $d_i$ so that, in addition to intersecting $D$ in
$R$ with multiplicity~8, $C$ has
contact multiplicity~3 with a vertical tangent line $L_P$ in a $\Q$-rational point $P$.
This forces certain conditions on the $d_i$ to be satisfied; these are spelt out in 
Lemma~\ref{lemma:quartic_vertical}.

\begin{ex}\label{ex:qcl1}
Let $a, d_1, d_4\in \Q$ and let
\[d_2 = \frac{3}{2}a^4-a^3d_1\;\;\textrm{and}\;\;d_3 = \frac{3}{4}a^5-d_4a^3.\] 
Then, if the discriminant of the plane quartic $C:F=0$ is nonzero, where $F$ is as in~\eqref{eq:fcq}, 
the conditions of Lemma~\ref{lemma:quartic_vertical} are satisfied for the point
 $P=(a^3, -a^4)\in C(\Q)$. Namely, $C$ is smooth and the vertical line
 $L_P : x=a^3$ intersects $C$ in $P$ with multiplicity~3 and in $\infty$ with multiplicity~1.
 By~Lemma~\ref{lemma:quartic_vertical} and Lemma~\ref{lemma:qc},
 the following are elements of $K_2^T(C)/\mathrm{torsion}$:
\[
    \{\infty, O, P\}, \{\infty, O, R\}, \{\infty, P, R\}
\]
Here the discriminant of $C$ factors as
\[
  a^{42}(3a^2-4d_4)^2(3a-2d_1)^8(13a^2-4ad_1-4d_4)^3q(a,d_1,d_4),
\]
where $q$ is homogeneous of degree~12 if we endow $a$ and $d_1$ with weight~1 and $d_4$
with weight~2.
\end{ex}

In fact we can do better: We can find a subfamily of the family from Example~\ref{ex:qcl1}
in two parameters such that there are two vertical tangent lines $L_{P_1}$ and $L_{P_2}$
intersecting $C$ in $\Q$-rational points $P_1$ and $P_2$, respectively, with contact
multiplicity~3. 
To this end, we simply have to solve a system of linear equations in $d_1,\ldots,d_4$,
given by the conditions of Lemma~\ref{lemma:quartic_vertical}.
\begin{ex}\label{ex:qcl2}
Let $a_1,\, a_2\in \Q$ and set
\[
    d = a_1^2+a_1a_2+a_2^2,\;
    d_1 = \frac{3}{2d}(a_1^3+a_1^2a_2+a_1a_2^2+a_2^3), \;
    d_2 = -\frac{3}{2d}a_1^3a_2^3, 
\]
\[d_3 = -\frac{3}{4d}a_2^3a_1^3(a_1+a_2),\;
    d_4 =
    \frac{3}{4d}(a_1^4+a_1^3a_2+a_1^2a_2^2+a_1a_2^3+a_2^4).
\]
Let $F$ be as in~\eqref{eq:fcq} and let $C$ be the projective closure of the affine curve
given by $F=0$.
Then, if $\disc(C)\ne 0$, the vertical line $L_{P_i} : x=a_i^3$ is tangent to $C$ in the
point $P_i = (a_i^3,-a_i^4) \in C(\Q)$, where $1\le i\le 2$.
In this case, the following are elements of $K_2^T(C)/\mathrm{torsion}$:
\[
  \{\infty, O, {P_1}\}, \{\infty, O, R\}, \{\infty, R, {P_1}\},
  \{\infty, O, {P_2}\}, \{\infty, R, {P_2}\}
\]
The discriminant of $C$ factors as
\begin{align*}
  \disc(C) =&  -\frac{43046721a_1^{42}a_2^{42}}{67108864(a_1^2+a_1a_2+a_2^2)^{22}}(4a_1^3+8a_1^2a_2+12a_1a_2^2+3a_2^3)^3(a_1-a_2)^4(a_2+a_1)^2\\
  &(a_1^8+22a_1^7a_2+67a_1^6a_2^2+140a_1^5a_2^3+161a_1^4a_2^4+140a_1^3a_2^5+67a_1^2a_2^6+22a_1a_2^7+a_2^8)\\
  &(3a_1^3+12a_1^2a_2+8a_1a_2^2 +4a_2^3)^3.
\end{align*}
\end{ex}

Finally, we can also combine Lemma~\ref{lemma:qc} with Theorem~\ref{thm:spq_fam1} and choose the
$d_i$ in such a way that $C$ has contact multiplicity~3 with the tangent line to $C$ in a
$\Q$-rational point $P$ (resp. $Q$)
and such that this tangent line also intersects $C$ in $\infty$ (resp. $O$).
This leads to the following family in two parameters.
\begin{ex}\label{ex:qcl3}
Let $a,b \in \Q$ such that
  \[ d_1 = b^3+\frac{3}{2}a,\; d_2 = -a^3b^3,\; d_3 = 2a^3(2b^3+3a)b^3,\; d_4 =
-4b^6-6ab^3+\frac{3}{4}a^2; \]
Let $F$ be as in~\eqref{eq:fcq} and let $C$ be the projective closure of the affine curve
given by $F=0$.
If 

\begin{align*}
  \disc(C) = &
-4a^{42}b^{42}(-8b^{18}-36ab^{15}-18a^2b^{12}+189a^3b^9+351a^4b^6+162a^5b^3+27a^6)\\
&(2b^3+3a)^2 (144b^{12}+576ab^9+504a^2b^6+112a^3b^3+9a^4)^2\\
\ne &0,
\end{align*}
then $C$ is smooth and the following are elements of $K_2^T(C)/\mathrm{torsion}$
\[
    \{\infty, O, P\}, \{\infty, O, R\}, \{\infty, R, P\},
    \{\infty, O, Q\}, \{\infty, R, Q\}, \{\infty, P, Q\},
\]
where $P= (a^3,-a^4) \in C(\Q)$  and $Q = (-a^2b^3,-a^2b^6) \in C(\Q)$.
See Figure~\ref{fig:qex2} on page~\pageref{fig:qex2} for the curve 
$C:x^4 + 1/64((2x - 8y - 1)^2 - 52x^2 + 8x)y=0$, corresponding to the choice
$a=1/2,b=-1$.
\end{ex}

\begin{rk}
    It would be interesting try to find~3 integral elements in the families discussed
    above.
    We have not checked under which conditions the elements constructed in
    the present section are integral, as our focus
    is on the geometric picture: One can force a curve $C$ to have specific intersection
    properties with other curves of varying degrees to produce elements of $K_2^T(C)/\mathrm{torsion}$.
\end{rk}
\begin{figure}[h]
  \begin{center}
    \scalebox{.6}{\includegraphics[width=\textwidth]{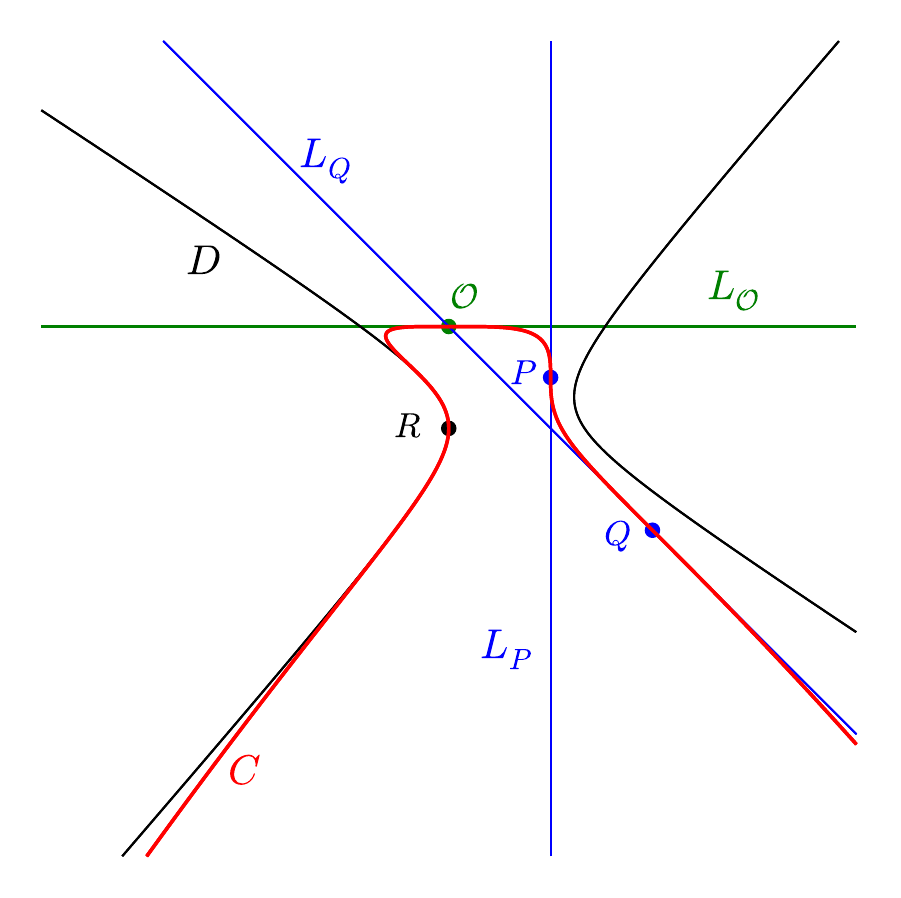}}
    \caption{The curve $C:x^4 + 1/64((2x - 8y - 1)^2 - 52x^2 + 8x)y=0$}
  \label{fig:qex2}
  \end{center}
\end{figure}

\section{Nekov\'a\v{r}-type constructions}\label{sec:nekovar}
In \cite{SR} a family of elliptic curves $E_a$ equipped with  an element in
$K_2^T(E_a)\otimes \Q$ is studied. The remarkable fact about this family is that this
element corresponds to a non-torsion divisor. The authors of~\cite{SR} attribute this
family to J. Nekov\'a\v{r}. We will recover this and similar construction in our setup. The
theoretical basis of such constructions can be phrased geometrically, as in the following lemma.  
In its formulation we denote, by abuse of
notation as before, the defining polynomial of a plane curve by the same letter as the curve itself.

\begin{lemma}\label{lemma:nekovar}  
Let $C$ be a {smooth projective} plane curve defined over $\Q$.
We assume that there are 3 pairwise distinct plane curves $E, G$ and $H$,
defined over
$\Q$, of the same degree and having pairwise distinct intersection with
$C$, which define functions $g$ and $h$
on $C$ via the quotients of their defining polynomials as $ g:= G/E$ and $h:=H/E$,
 satisfying the following properties: 
\begin{itemize}
\item There is a rational point $\infty$ such that  $C$ and $E$ have maximal contact in $\infty$.
\item For all $P$ in the intersection of $G$ and $C$ the divisor $(P)-(\infty)$ is a torsion divisor, i.e. if 
\[
  C \cap G = \sum a_i P_i
\]
then there exist functions $g_i$ and integers $m_i$ with
\[
 \div (g_i) = m_i  ( P_i) - m_i(\infty).
 \]
\item There is a constant $\kappa \in \Q^\times$ such that for all $Q$ in the intersection of $H$ and $C$ the
   value of $g$ at $Q$ equals $\kappa$.
\end{itemize}

Then replacing $g$ by $\tilde{g} := g/\kappa$ and setting $\kappa_i =  T_{P_i}(\{\tilde{g},h\}) $, we get the element
\begin{align}\label{eq:neko-elem}
   m \{\tilde{g},h\}   - \sum_i \frac{m}{m_i} \{ \kappa_i , g_i\}
\in K^T_2(C), 
\end{align}
where  $m=\operatorname{lcm}(\{ m_i\})$ is the least common multiple of the multiplicities $m_i$.
 \end{lemma}

\begin{proof} It suffices to prove that the tame symbol equals $1$ at all 
critical points except $\infty$; because of the product formula \eqref{eq:prod_form} 
the claim then follows.
  At each  $Q$ in the support of $H \cap C$ we have $T_{Q}(\{\tilde{g},h\} ) =1$ by construction 
  and $T_{Q}(\{\kappa_i,g_i\}) =1$ by assumption. 
For each $P_i$  in the support of 
$G\cap C$ we get   $T_{P_i}(\{\kappa_j , g_j\}) = 1$  if $i\neq j$ and else
\[
T_{P_i}(\{\kappa_i , g_i\}) = (-1)^{\ord_{P_i}(\kappa_i)\ord_{P_i}(g_i)}
    \frac{\kappa_i^{\ord_{P_i}(g_i)}}{g_i^{\ord_{P_i}(\kappa_i)}}(P_i)
    = 
    \frac{\kappa_i^{m_i}}{g_i^{0}}(P_i) = \kappa_i^{m_i} =   T_{P_i}( m_i \,\{\tilde{g},h\})
\]
\end{proof}

Observe that by Galois descent the  Nekov\'a\v{r} element in \eqref{eq:neko-elem} 
is defined over $\Q$, although the $P_i$ or the $Q_j$ might not be rational points.
If we divide the Nekov\'a\v{r} element in \eqref{eq:neko-elem}  by $m$  
we get an element in $K_2(C)\otimes \Q$ instead. Thus in the case of elliptic curves we recover
the construction 5.1 of \cite{SR}. We do not study the integrality of such elements in our
examples, instead we refer the interested reader to  \cite[\S5.2]{SR} for the  case of elliptic curves.

\begin{rk}
The assumptions of the Lemma can be weakenend as follows: Firstly, we do not have
to assume that the curves $E,G$ and $H$ have the same degree, because we can replace the
equation of a curve by some power, if needed. 
 Secondly, if the values of the function 
$g$ at the $Q_j$ differ by some root of unity, then we just need to replace $g$ by some power. 
Thirdly, the point $\infty$ can be a singular point (but there are no other singular
points).
In this case the condition on $F$ means that on the normalization of $C$, the zero divisor
of $g_i$ is a multiple of $(P_i)$ and the pole divisor is supported in the points mapping
to $\infty$. 
If there is only one such point, then the requirements of Remark~\ref{rk:prod_sing} are
satisfied, and the proof of Lemma~\ref{lemma:nekovar} goes through, yielding an element in
$K_2^T\otimes\Q$ of the normalization of $C$.
If there are several points $\infty_1,\ldots,\infty_m$ above $\infty$ in the normalization
of $C$, Remark~\ref{rk:prod_sing} does not apply, but if we can
show that for the elements in question, the tame symbol at all $\infty_i$ is~1, then the conclusion of the Lemma still holds.
\end{rk}

\begin{ex}\label{ex:ell_2tor} 
We want to apply the Lemma \ref{lemma:nekovar} to a family of elliptic curves.
We let $E$ be the tangent line at $\infty$ and for $G$ we take the line $y=0$; 
the points $P_i$ are therefore the affine $2$-torsion points $(x_i,0)$ and the functions $g_i$ are
given by the vertical tangent lines $x-x_i$ in the points $P_i$. 

Consider the family
\[C_r: {y}^{2} = {x}^{3} + \left( -\frac{1}{3}+\frac{2}{3}\,r-\frac{4}{3}\,{r}^{2} \right) x + {\frac {2}{27}}-\frac29\,r-\frac59\,{r}^{2}+{\frac {16}
{27}}\,{r}^{3} ,\]
where we exclude the finite set of those $r \in \Q$ for which the curve becomes singular.
Then for each $r \in \Q$ the line
\[H_r : y = -x+\frac{1}{3} -\frac{1}{3} r
\]
meets $C_r$ in the points $Q_1 = ( -\frac{4}{3} r + \frac{1}{3},   r) $ and $Q_2 = (\frac{2}{3} r
+\frac{1}{3} , -r)$ to order $1$ and
$2$, respectively. Finally, since 
$| y(Q_1)|= |y(Q_2)|=r$ the assumptions of Lemma \ref{lemma:nekovar} are satisfied.
See Figure~\ref{fig:ell_ex} for the case $r=3/4$.
  Note that the geometric configuration is the same as in the family of elliptic curves
  due to Nekov\'a\v{r} and discussed in~\cite{SR}. In particular, that family can also be
  constructed using our approach.
  Nevertheless, one can show that it is not isomorphic over $\Q$ to the family $\{C_r\}$.
\begin{figure}[h]
  \begin{center}
    \scalebox{.6}{\includegraphics[width=\textwidth]{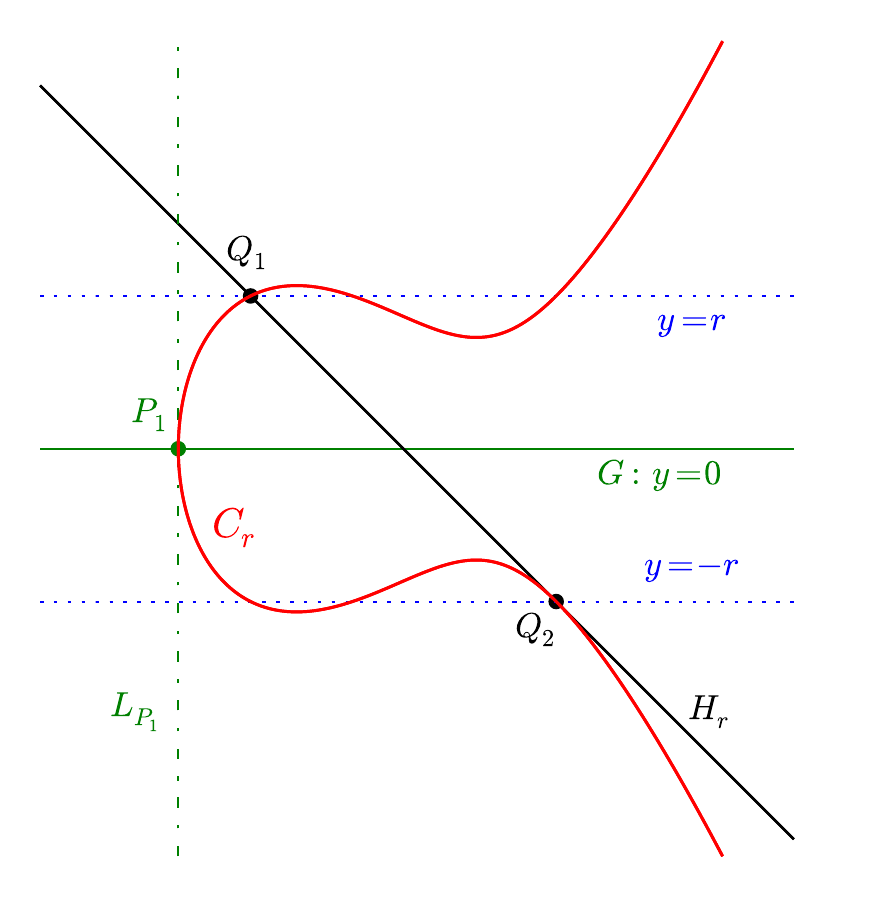}}
    \caption{The curve $C_r$ from Example~\ref{ex:ell_2tor} for $r=3/4$}
  \label{fig:ell_ex}
  \end{center}
\end{figure}
\end{ex}

\begin{ex}\label{ex:ell_3tor}
Lemma~\ref{lemma:nekovar} also applies to the family of elliptic curves  given by 
\[
  C_r: y^2 + f_1(x)y +x^3 = 0,
 \]
 where $f_1(x) = r - (4r+1)x$ and  
where $r \in\Q\setminus\{0,-1, 1/8\}$.
There is only one point $\infty$ at infinity and its tangent to $C$ has maximal contact for all $r$.
The line $G$ given by  $y=0$ meets $C$ only in the origin $(0,0)$, which is therefore a
$3$-torsion point and we can take $g_1 = y$. 
Finally the line $H:  y=x-r$ intersects $C$ in exactly two points $Q_1= (0,-r)$ and
$Q_2=(2r,r)$.
\end{ex}

\begin{ex}\label{ex:hyp_nek}
We can extend the approach of
Example~\ref{ex:ell_3tor} to curves genus $g=2$ (we have not attempted to treat
hyperelliptic curves of higher genus).
This means that we want our curve to have a point $O$ of maximal contact multiplicity,
so we are essentially in the situation discussed in detail in Section~\ref{sec:hyp}.

In order to construct our family, we work with affine models of the form
\begin{equation}\label{eq:2inf_model}
   y^2 + f_1(x)y +x^5 = 0,
\end{equation}
where $f_1$ has degree at most~3. Note that in Section~\ref{sec:hyp}, we required
$\deg(f_1) \le 2$ in the case of curves of genus~2, which guarantees that there is a unique point at
infinity on a smooth model.
Geometrically, we again work on the singular model $C'$ defined as the projective closure of the
affine curve given by~\eqref{eq:2inf_model}.

In the language of Lemma~\ref{lemma:nekovar}, we let $E=L_\infty$ be the line intersecting $C'$ in
the point $\infty$ at infinity with maximal contact multiplicity.
The point $O = (0,0)$ is a point of maximal contact between $C'$ and the
tangent line $G=L_O :\, y=0$ to $C'$ in $O$, so we take $g_1 = y$.

In order to apply Lemma~\ref{lemma:nekovar}, we want a line $H$, which intersects $C'$ in exactly two points $Q_1$ and
$Q_2$, whose $y$-coordinates have absolute value equal to the same rational number $r$.
For instance, we can require that $Q_1$ is a~3-contact point and that $Q_2$ is a~2-contact
point.
Rewriting this condition as a system of equations, we find that there is no such example
if $\deg(f_1)\le2$. Hence we have to allow $\deg(f_1)=3$.
We find the family of genus~2 curves $C_r$ given by~\eqref{eq:2inf_model} with
\[
  f_1 = r-x+4rx^3;
\]
these curves are smooth unless $r \in \{0,1/3\}$.
Here, in analogy with  Example~\ref{ex:ell_3tor}, the  line $H_r:   y = x - r $ meets $C$ in the points $Q_1= (0,-r)$ and $Q_2=(2r,r)$ 
with multiplicities $3$ and $2$ respectively.
See Figure~\ref{fig:g2_ex} for the case $r=1/2$. 

\begin{figure}[h]
  \begin{center}
    \scalebox{.6}{\includegraphics[width=\textwidth]{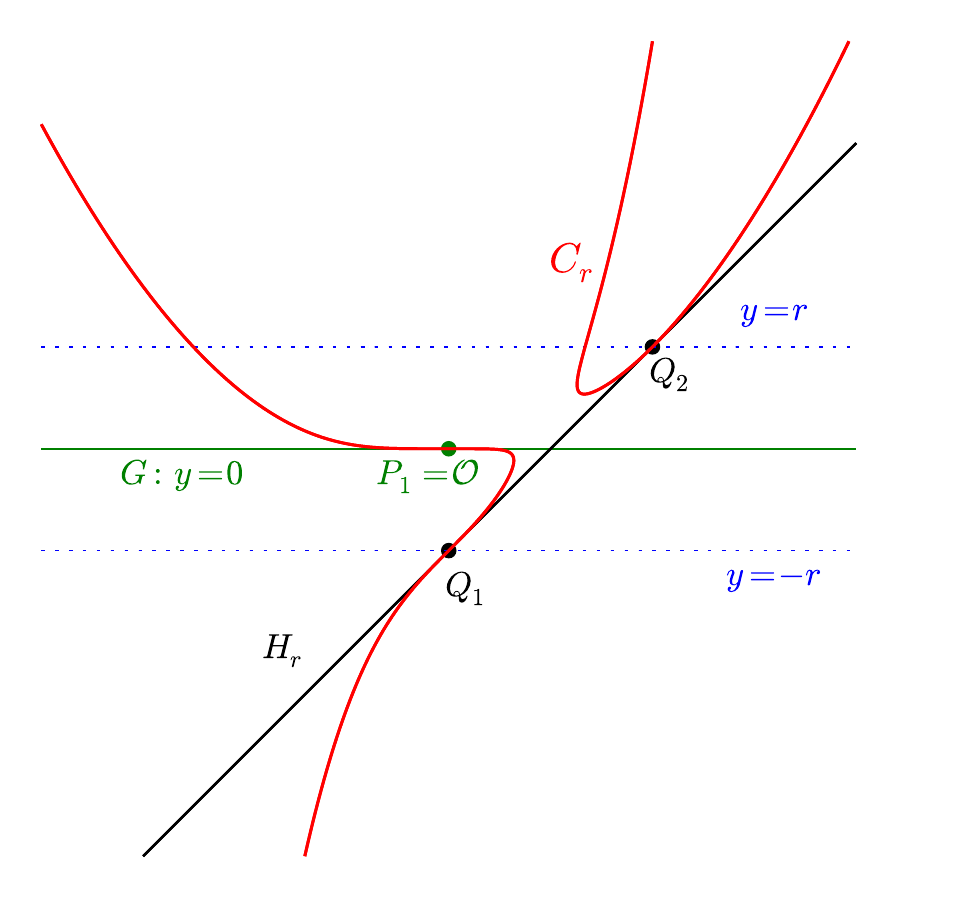}}
    \caption{The curve $C_r$ from Example~\ref{ex:hyp_nek} for $r=1/2$}
  \label{fig:g2_ex}
  \end{center}
\end{figure}

It remains to check that we actually get an element of $K_2^T(C_r)$.
The only potential problem is that we need to check that the tame symbols at the two points at
infinity on a nonsingular model of $C_r$ are equal to~1, but in our specific situation
this is easily seen to hold by a simple computation once we have scaled $h_r$ by $r^{-1}$. 
\end{ex}

\bibliography{k2-biblio}

\begin{thebibliography}{DdJZ06}

\bibitem[Bas68]{bass}
Hyman Bass.
\newblock {\em Algebraic {$K$}-theory}.
\newblock W. A. Benjamin, Inc., New York-Amsterdam, 1968.

\bibitem[BCP97]{magma}
Wieb Bosma, John Cannon, and Catherine Playoust.
\newblock The {M}agma algebra system. {I}. {T}he user language.
\newblock {\em J. Symbolic Comput.}, 24(3-4):235--265, 1997.
\newblock Computational algebra and number theory (London, 1993).

\bibitem[Bei85]{beilinson}
A.A. Beilinson.
\newblock Higher regulators and values of {$L$}-functions.
\newblock {\em Journal of Soviet Mathematics}, 30(2):2036--2070, 1985.

\bibitem[BG86]{bloch-grayson}
S.~Bloch and D.~Grayson.
\newblock {$K_2$} and {$L$}-functions of elliptic curves: computer
  calculations.
\newblock In {\em Applications of algebraic {$K$}-theory to algebraic geometry
  and number theory, {P}art {I}, {II} ({B}oulder, {C}olo., 1983)}, volume~55 of
  {\em Contemp. Math.}, pages 79--88. Amer. Math. Soc., Providence, RI, 1986.

\bibitem[DdJZ06]{ddz}
Tim Dokchitser, Rob de~Jeu, and Don Zagier.
\newblock Numerical verification of {B}eilinson's conjecture for {$K_2$} of
  hyperelliptic curves.
\newblock {\em Compos. Math.}, 142(2):339--373, 2006.

\bibitem[Den89]{deninger}
Christopher Deninger.
\newblock Higher regulators and {H}ecke {$L$}-series of imaginary quadratic
  fields. {I}.
\newblock {\em Invent. Math.}, 96(1):1--69, 1989.

\bibitem[Dix87]{dixmier}
J.~Dixmier.
\newblock On the projective invariants of quartic plane curves.
\newblock {\em Adv. in Math.}, 64(3):279--304, 1987.

\bibitem[dJ05]{deJeu:limit}
Rob de~Jeu.
\newblock A result on {K2} of certain (hyper)elliptic curves.
\newblock {\em Preprint}, 2005.
\newblock \url{http://www.few.vu.nl/~jeu/abstracts/reg-estimate.dvi}.

\bibitem[Gar71]{garland}
Howard Garland.
\newblock A finiteness theorem for {$K_{2}$} of a number field.
\newblock {\em Ann. of Math. (2)}, 94:534--548, 1971.

\bibitem[Koh]{echidna}
David Kohel.
\newblock Echidna: Algorithms for elliptic curves and higher dimensional
  analogues.
\newblock Available at
  \url{http://echidna.maths.usyd.edu.au/kohel/alg/index.html}.

\bibitem[LdJ15]{dl}
Hang Liu and Rob de~Jeu.
\newblock On {$K_2$} of certain families of curves.
\newblock {\em To appear in IMRN}, 2015.
\newblock \url{http://arxiv.org/abs/1402.4822}.

\bibitem[Liu02]{liu}
Qing Liu.
\newblock {\em Algebraic geometry and arithmetic curves}, volume~6 of {\em
  Oxford Graduate Texts in Mathematics}.
\newblock Oxford University Press, Oxford, 2002.
\newblock Translated from the French by Reinie Ern{\'e}, Oxford Science
  Publications.

\bibitem[Mil71]{milnor}
John Milnor.
\newblock {\em Introduction to algebraic {$K$}-theory}.
\newblock Princeton University Press, Princeton, N.J., 1971.
\newblock Annals of Mathematics Studies, No. 72.

\bibitem[Qui73]{quillen}
Daniel Quillen.
\newblock Higher algebraic {$K$}-theory. {I}.
\newblock In {\em Algebraic {$K$}-theory, {I}: {H}igher {$K$}-theories ({P}roc.
  {C}onf., {B}attelle {M}emorial {I}nst., {S}eattle, {W}ash., 1972)}, pages
  85--147. Lecture Notes in Math., Vol. 341. Springer, Berlin, 1973.

\bibitem[RS98]{SR}
Klaus Rolshausen and Norbert Schappacher.
\newblock On the second {$K$}-group of an elliptic curve.
\newblock {\em J. Reine Angew. Math.}, 495:61--77, 1998.

\bibitem[Sal60]{salmon}
George Salmon.
\newblock {\em A treatise on the higher plane curves: intended as a sequel to
  ``{A} treatise on conic sections''}.
\newblock 3rd ed. Chelsea Publishing Co., New York, 1960.

\bibitem[Sch88]{schneider}
Peter Schneider.
\newblock Introduction to the {B}eilinson conjectures.
\newblock In {\em Beilinson's conjectures on special values of
  {$L$}-functions}, volume~4 of {\em Perspect. Math.}, pages 1--35. Academic
  Press, Boston, MA, 1988.

\bibitem[Ver]{vermeulen}
Alexius~Maria Vermeulen.
\newblock {\em Weierstrass points of weight two on curves of genus three}.
\newblock Dissertation, Universiteit van Amsterdam, 1983.

\end{thebibliography}
\bibliographystyle{alpha}

\end{document}